\documentclass[a4paper,10pt,reqno]{amsart}
\numberwithin{equation}{section}
\usepackage{latexsym,amssymb}
\usepackage{hyperref}
\usepackage{amsmath,amsthm}
\usepackage{amsfonts}
\usepackage[american]{babel}
 \usepackage[T1]{fontenc}
 \usepackage[utf8]{inputenc}
\usepackage{mathrsfs}
\usepackage{psfrag}
\usepackage{epsfig,inputenc}
\usepackage{graphpap,latexsym,epsf}
\usepackage{amssymb,eucal,paralist,enumerate}
\usepackage{graphicx}

\usepackage{tgheros}

\usepackage[eulergreek]{sansmath}

\usepackage{tikz} 
\usetikzlibrary{matrix}
\usepackage{rotating}

\newcommand{\bbM}{\mathbb{M}}

\newcommand{\R}{\mathbb{R}}
\newcommand{\N}{\mathbb{N}}
\newcommand{\C}{\mathbb{C}}

\mathchardef\emptyset="001F

\numberwithin{equation}{section}

\newtheorem{theorem}{Theorem}[section]

\newtheorem{lemma}[theorem]{Lemma}
\newtheorem{remark}[theorem]{Remark}

\newtheorem{definition}[theorem]{Definition}
\newtheorem{proposition}[theorem]{Proposition}

\newtheorem{notation}[theorem]{Notation}
\newtheorem{corollary}[theorem]{Corollary}

\newcommand{\eps}{\varepsilon}

\newcommand{\weakto}{\rightharpoonup} 

\newcommand{\aein}{\text{a.e.\ in }}

\newcommand{\down}{\downarrow}
\newcommand{\weaksto}{\overset{*}{\rightharpoonup}}

\newcommand{\AC}{\mathrm{AC}}

\newcommand{\dualoperator}

 \def\calE{{\mathcal E}} 
 \def\calH{{\mathcal H}} 
 \def\calK{{\mathcal K}} 
\def\calM{{\mathcal M}}  
 \def\calQ{{\mathcal Q}} \def\calR{{\mathcal R}}
\def\calS{{\mathcal S}}  
\def\calV{{\mathcal V}} \def\calW{{\mathcal W}} 
 
 \def\rmn{{\mathrm n}}

  \def\rmC{{\mathrm C}}
\def\rmD{{\mathrm D}} \def\rmE{{\mathrm E}}

\def\FG{\mathbf}

 \def\bfQ{{\FG Q}}

\def\dd{\;\!\mathrm{d}} 

\newcommand{\pairing}[4]{ \sideset{_{ #1 }}{_{ #2 }}  {\mathop{\langle #3 , #4
\rangle}}}

\newcommand{\nchi}{{\raise.2ex\hbox{$\chi$}}}

\newcommand{\piecewiseConstant}[2]{\overline{#1}_{\kern-1pt#2}}

\newcommand{\upiecewiseConstant}[2]{\underline{#1}_{\kern-1pt#2}}

\newcommand{\piecewiseLinear}[2]{{#1}_{\kern-1pt#2}}

\newcommand{\pwwll}[2]{\widehat{#1}_{\kern-1pt#2}}

\newcommand{\piecewiseVariational}[2]{\tilde{#1}_{\kern-1pt#2}}

\newcommand{\DDDn}[2]{\begin{array}[t]{c}#1\vspace*{-1em}\\_{#2}\end{array}}

\newcommand{\dddn}[2]{\DDDn{\begin{array}[t]{c}\underbrace{#1}\vspace*{.6em}\end{array}}{\text{\footnotesize #2}}}

\newcommand{\foraa}{\text{for a.a.}}

\newcommand{\BV}{\mathrm{BV}}
\newcommand{\BD}{\mathrm{BD}}

\newcommand{\Dir}{\mathrm{Dir}}
\newcommand{\Neu}{\mathrm{Neu}}

 \def\trait #1 #2 #3 {\vrule width #1pt height #2pt depth #3pt}

 \def\fin{\hfill
         \trait .3 5 0
         \trait 5 .3 0
         \kern-5pt
         \trait 5 5 -4.7
         \trait 0.3 5 0
 \medskip}

\newcommand{\bbC}{\mathbb{C}}

\newcommand{\bbD}{\mathbb{D}}

\newcommand{\mt}{\bbM}

\newcommand{\sym}{\mathrm{sym}}
\newcommand{\dev}{\mathrm{D}}

\newcommand{\dip}[2]{\mathrm{H}(#1,#2)}
\newcommand{\dipname}{\mathrm{H}}

\newcommand{\did}[1]{\mathrm{R}(#1)}

\newcommand{\didname}{\mathrm{R}}

\newcommand{\Gdir}{\Gamma_{\Dir}}
\newcommand{\Gneu}{\Gamma_{\Neu}}

\newcommand{\As}{A_{\mathrm{m}}}
\newcommand{\Hs}{H^{\mathrm{m}}}
\newcommand{\ass}{a_{\mathrm{m}}}

\newcommand{\disv}[2]{\calV_{#1,#2}}

\newcommand{\psin}[1]{\Psi_{#1}}
\newcommand{\psie}[2]{\Psi_{#1,#2}}

\newcommand{\Mnn}{{\mathbb{M}^{n\times n}_{\mathrm{sym}}}}
\newcommand{\MD}{{\mathbb M}^{n{\times}n}_{\mathrm{D}}}
\newcommand{\Sym}{\mathrm{Sym}}
\newcommand{\Lin}{\mathrm{Lin}}
\newcommand{\Mb}{\mathrm{M}_{\mathrm{b}}}
\newcommand{\MbD}{{\Mb(\Omega \cup \Gdir; \MD)}}
\newcommand{\Lnn}{{L^2(\Omega; \Mnn)}}
\newcommand{\Linftyn}{{L^\infty(\Omega; \MD)}}

\newcommand{\DVito}{\mathcal{D}}
\newcommand{\DVitos}{\mathcal{D}_\nu^{*,\mu}}

\newcommand{\DVitoskk}{\mathcal{D}_{\nu_k}^{*,\mu_k}}

\newcommand{\DVitosred}{\mathcal{D}^{*,\mu}}

\newcommand{\diver}{\mathrm{div}}
\newcommand{\Diver}{\mathrm{Div}}
\newcommand{\hn}{\mathscr{H}^{n-1}}

\newcommand{\ol}{\overline}

\newcommand{\sig}[1]{\mathrm{E}(#1)}

\newcommand{\sft}{\mathsf{t}}
\newcommand{\sfu}{\mathsf{u}}
\newcommand{\sfz}{\mathsf{z}}
\newcommand{\sfp}{\mathsf{p}}
\newcommand{\sfe}{\mathsf{e}}
\newcommand{\sfs}{\mathsf{s}}
\newcommand{\sff}{\mathsf{f}}
\newcommand{\serifsigma}{{\sansmath \sigma}}
\newcommand{\sfq}{\mathsf{q}}
\newcommand{\Me}[4]{\calM_{\eps}^{\mu,\nu}(#1,#2,#3,#4)}

\newcommand{\Mename}{\calM_{\eps}^{\mu,\nu}}

\newcommand{\Mredname}{\calM_{\eps,\mathrm{red}}^{\mu, \nu}}

\newcommand{\Mli}[4]{\calM_{0}^{\mu,\nu}(#1,#2,#3,#4)}

\newcommand{\Mliname}{\calM_{0}^{\mu,\nu}}

\newcommand{\Mliredname}{\calM_{0,\mathrm{red}}^{\mu,\nu}}

\newcommand{\oMliname}{\calM^{0,0}_0}

\newcommand{\oMliredname}{\calM_{0,\mathrm{red}}^{0,0}}
\newcommand{\Mlizero}[4]{\calM_{0}^{\mu,0}(#1,#2,#3,#4)}
\newcommand{\Mlinamezero}{\calM_{0}^{\mu,0}}
\newcommand{\Mliredzero}[8]{\calM_{0,\mathrm{red}}^{\mu,0}(#1,#2,#3,#4,#5,#6,#7, #8)}
\newcommand{\Mlirednamezero}{\calM_{0,\mathrm{red}}^{\mu,0}}

\newcommand{\tilded}{\widetilde{d}}
\newcommand{\dLtwo}{d_{L^2}}
\newcommand{\enen}[1]{\calE_{#1}}
\newcommand{\newmu}{{ \mu}}
\newcommand{\adm}[5]{\mathcal{A}(#1,#2; [#3,#4]{\times}#5)}
\newcommand{\eadm}[5]{\mathcal{EA}(#1,#2; [#3,#4]{\times}#5)}

\newcommand{\Qpp}{\bfQ_{\scriptstyle \tiny \mathrm{PP}}}
\newcommand{\Qha}{\bfQ_{\scriptstyle \tiny \mathrm{H}}}
\newcommand{\Hpp}{\mathcal{H}_{\scriptstyle \tiny \mathrm{PP}}}

\newcommand{\CASE}[1]{\noindent\underline{\emph{Case #1}}}

\newcommand{\BVZ}{\BV_0}

\newcommand{\BVZZZ}{\BV_0^{0,0}}

\newcommand{\BVA}[1]{\BV_{0}^{#1,0}}
\newcommand{\BVB}[1]{\BV_0^{#1}}

\begin{document}
\title[Singular limits of a  system for damage and plasticity]{Singular limits  of a  coupled elasto-plastic damage system 
\smallskip
\\
as viscosity and hardening vanish}

\author{Vito Crismale, Giuliano Lazzaroni  \and Riccarda Rossi}

\address{V.\ Crismale, Dipartimento di Matematica ``Guido Castelnuovo'', Sapienza Universit\`a di Roma, Piazzale Aldo Moro 2, I-00185 Roma, Italy}
\email{vito.crismale\,@\,uniroma1.it}

\address{G.\ Lazzaroni, Dipartimento di Matematica e Informatica ``Ulisse Dini'',
Universit\`a degli Studi di Firenze, Viale Morgagni 67/a, 50134 Firenze, Italy}
\email{giuliano.lazzaroni\,@\,unifi.it}

\address{R.\ Rossi, DIMI, Universit\`a degli studi di Brescia,
via Branze 38, 25133 Brescia, Italy}
\email{riccarda.rossi\,@\,unibs.it}

\thanks{The authors have been funded by the Italian Ministry of University and Research through two different projects:
MIUR - PRIN project
2017BTM7SN 
\emph{Variational Methods for stationary and evolution problems with singularities and interfaces},
MIUR - PRIN project 2017TEXA3H \emph{Gradient Flows, Optimal Transport and Metric Measure Structures}. 
G.L.\ and R.R.\ have been partially supported by the  Gruppo Nazionale per  l'Analisi Matematica, la  Probabilit\`a  e le loro Applicazioni (GNAMPA) of the Istituto Nazionale di Alta Matematica (INdAM)}


\begin{abstract} 
The paper studies the asymptotic analysis of a model coupling elastoplasticity and damage depending on three parameters -- governing viscosity, plastic hardening, 
and convergence rate of plastic strain and displacement to equilibrium -- as they vanish in different orders. 
The notion of limit evolution obtained is proven to coincide in any case with a notion introduced by Crismale and Rossi in \cite{Crismale-Rossi};
 moreover, such solutions are closely related to those obtained in the vanishing-viscosity limit by Crismale and Lazzaroni in \cite{Crismale-Lazzaroni}, for the analogous model where only the viscosity parameter was present.
\end{abstract}
\maketitle

\medskip

\section{Introduction}
Rate-independent processes model evolutionary phenomena where the external loading is much slower than the internal oscillations of materials, while viscosities may be neglected. Despite a wide literature on the subject (see \cite{MieRou15} and references therein), a further understanding is needed of the relations between the  different notions of solution that have been proposed. In particular, a problem of interest in applications is to determine which kind of solution captures the limiting behavior of dynamical systems for small viscosity or inertia. For this reason, in this paper we compare different notions of solution, obtained with different approximation methods as viscosities tend to zero at different rates (but with no inertia). 
\par
%
We focus  on  a    rate-independent system modeling damage in an elasto-plastic body occupying  a bounded  Lipschitz domain $\Omega\subset \R^n$, $n\geq 2$. The model was advanced and first studied in \cite{AMV14, 
AMV15}, 
while the existence of \emph{globally minimizing  quasistatic evolutions} (or, equivalently, \emph{Energetic solutions}) was first proved in 
\cite{Crismale}. In \cite{Crismale-Lazzaroni, Crismale-Rossi}, the  \emph{vanishing-viscosity} approach was instead adopted to  find  the so-called 
\emph{Balanced Viscosity} solutions,  obtaining the rate-independent system as the limit of a viscously perturbed system. 
  We refer to the pioneering \cite{EfeMie06RILS}, and the subsequent
\cite{MRS12,  MRS13}, for the definition and properties of such solutions in the context of an `abstract' rate-independent system.  The vanishing-viscosity technique has also  been adopted in various concrete applications, 
 ranging from plasticity  (cf., e.g., \cite{DMDMM08, DalDesSol11, BabFraMor12, FrSt2013, Sol14}), 
	 to damage,  fracture, and fatigue  (see for instance \cite{KnMiZa08ILMC, Lazzaroni-Toader, KRZ13, Almi17, CL17fra, ACO2018}).
\par
In this paper we aim to gain further insight into  the  different ways of constructing  \emph{Balanced Viscosity} solutions to the model for damage and plasticity from \cite{AMV15}, which were
explored in \cite{Crismale-Lazzaroni} and \cite{Crismale-Rossi}.
We show that these notions of solutions essentially coincide if the hardening vanishes together with viscosities, while they retain different features if the hardening parameter is positive. In particular, it turns out that perfect plasticity coupled with damage may equivalently be approximated  by means of processes where viscosity is confined to the flow rule for damage, or with viscosity also in the momentum equation and in the plastic flow rule. 
\par 
\subsection*{The model}
The rate-independent process we are going to address
 describes the evolution, in the time interval $(0,T)$, of the  \emph{displacement} 
$u: (0,T) \times \Omega \to \R^n$, of the plastic strain  $p: (0,T)\times \Omega \to \MD$, and of the damage variable $z: (0,T)\times \Omega \to  [0,1]$ that describes  the soundness of the material: 
for $z(t,x)=1$ (respectively,  $z(t,x)=0$) the material is in the undamaged (fully damaged, resp.) state,  at the time $t\in (0,T)$ and `locally' around the point $x\in \Omega$. 
In fact, the related PDE system 
 consists of 
 \begin{subequations}
 \label{RIS-intro}
 \begin{itemize}
\item[-] the momentum balance
\begin{align}
\label{mom-balance-intro}
 - \mathrm{div}\,\sigma = f  \quad \text{ in } \Omega \times (0,T)\,, \qquad \quad  \sigma \rmn =g \text{ on } \Gneu \times (0,T), 
 \end{align}
 with $f,\,g$ some external forces,  $\rmn$ the outer unit normal vector to $\Omega$,  $\sigma$ the  \emph{stress} tensor
\begin{equation} \label{stress-intro}
 \sigma =  \mathbb{C}(z)e  
 \quad   \text{ in } \Omega \times (0,T), 
\end{equation}
$\mathbb{C}$ the elastic stress tensor,  and $e: (0,T)\times \Omega \to \Mnn$ the elastic strain; together  with the plastic strain $p$, the elastic strain $e$ concurs to 
  the kinematic admissibility condition for the \emph{strain} $ \rmE(u)  = \frac{\nabla u + \nabla u^T}{2}$, i.e.\
 \begin{equation}
 \label{kam-intro}
    \rmE(u)  = e+p   \quad  \text{ in } \Omega \times (0,T);
 \end{equation} 
 \item[-] the flow rule for the damage variable $z$
 \begin{align}
\label{flow-rule-dam-intro}
\partial\did{\dot{z}}  + \As (z)+ W'(z) \ni - \tfrac12\bbC'(z)e : e  \quad  \text{ in } \Omega \times (0,T),
\end{align}
where $\partial \didname  : \R \rightrightarrows \R$ denotes the convex analysis subdifferential of  the density of dissipation potential
\[
\didname:\R \to [0,+\infty]  \  \text{ defined by } \   \did{\eta}:= \left\{ \begin{array}{ll}
\kappa |\eta| & \text{ if } \eta\leq 0,
\\
+\infty & \text{ otherwise},
\end{array}
\right.
\]
encompassing the unidirectionality in the evolution of damage,  $\As$ is 
 the  $\mathrm{m}$-Laplacian operator, with $\mathrm{m}>\tfrac n2$, 
 and $W$ is a suitable nonlinear, possibly nonsmooth, function; 
  \item[-] the flow rule for the plastic tensor 
 \begin{align}
&
\label{flow-rule-plast-intro}
\partial_{\dot{p}}  \dip{z}{\dot{p}}  \ni \sigma_{\mathrm{D}}  \quad \text{ in } \Omega \times (0,T),
\end{align}
with $\sigma_\dev$ the deviatoric part of the stress tensor $\sigma$ and $\dipname(z, \cdot)$ the density of the plastic dissipation potential.
 \end{itemize}
System \eqref{mom-balance-intro}--\eqref{flow-rule-plast-intro} is complemented by 
the boundary conditions 
\begin{equation}
 \label{viscous-bound-cond}
 u=w  \text{ on } \Gdir \times (0,T),   \qquad 
 \sigma \rmn =g \text{ on }  \Gneu \times (0,T), \qquad \partial_{\rmn} z =0 \text{ on } \partial\Omega \times (0,T), 
  \end{equation}
where $\Gdir$ is   the Dirichlet part of the boundary $\partial \Omega$ and  $w$ a  time-dependent Dirichlet loading, while 
$\Gneu$ is the Neumann part of $\partial\Omega$ 
and $g$ an assigned traction.
 \end{subequations}
 \par
 Alternative models for damage and plasticity have been analyzed in, e.g., \cite{Roub-Valdman2016, RouVal17, DavRouSte19}, albeit from a different perspective. In fact, those papers address
the rate-independent evolution of the damage and plastic processes coupled with a \emph{rate-dependent} momentum balance, featuring viscosity and  even inertial terms. Therefore, the resulting system has a \emph{mixed} rate-dependent/independent character and is formulated in terms of a  weak, energetic-type notion of solution. Instead, both in  \cite{Crismale-Lazzaroni}  and \cite{Crismale-Rossi},
(two  distinct)
 viscous regularization procedures, described below, were advanced to construct Balanced Viscosity solutions to
 the fully rate-independent 
system \eqref{RIS-intro}. 
 \subsection*{Balanced Viscosity solutions} 
 In \cite{Crismale-Lazzaroni} 
 the vanishing-viscosity approximation of 
  system \eqref{RIS-intro} was carried out by perturbing  the 
 damage flow rule by a viscous term,
 which led to the \emph{viscously regularized} system 
 \begin{subequations}
 \label{viscous-CL-intro}
 \begin{align}
 & - \mathrm{div}\, \sigma = f   \quad    \text{with $\sigma = \bbC(z) e$} &&  \text{ in } \Omega \times (0,T), 
  \\
&  \partial\did{\dot{z}}  +\eps \dot z  + \As (z)+ W'(z) \ni - \tfrac12\bbC'(z)e : e  &&    \text{ in } \Omega \times (0,T),
\\
& \partial_{\dot{p}}  \dip{z}{\dot{p}}  \ni \sigma_{\mathrm{D}}  &&    \text{ in } \Omega \times (0,T), 
 \end{align}
{supplemented by the boundary conditions}
\begin{equation} 
\label{overall-BC}
 u=w  \text{ on } \Gdir \times (0,T),  \quad    \sigma \rmn =g \text{ on } \Gneu \times (0,T), \quad    \partial_\rmn  z =0 \text{ on } \partial\Omega  \times (0,T).
 \end{equation}
 \end{subequations}
 Passing to the limit in a reparameterized version of \eqref{viscous-CL-intro}
 led to a first construction of $\BV$ solutions to system  \eqref{RIS-intro}. 
 We shall illustrate the  notion of \emph{parameterized} Balanced Viscosity  solution thus obtained in the forthcoming Section 
  \ref{ss:3.1}.    In what follows, for quicker reference we will call the $\BV$ solutions from \cite{Crismale-Lazzaroni} 
  \emph{$\BVZ$ solutions to system  \eqref{RIS-intro}},
where the subscript 0 indicates that the solutions are obtained in the limit as $\eps\down0$.  
  \par
  In \cite{Crismale-Rossi}, a different construction of $\BV$ solutions 
  to system  \eqref{RIS-intro}
   was proposed, based on  a  viscous regularization of the momentum balance and of the plastic flow rule, \emph{in addition} to that of the damage flow rule.  This alternative  approach was first proposed in 
   \cite{MRS14} in a finite-dimensional context, and extended to infinite-dimensional systems in the recent \cite{MieRosBVMR}. Both papers address the vanishing-viscosity analysis of an abstract evolutionary  system, that can be thought of as prototypical of  rate-dependent systems  in  solid mechanics, 
   governing the evolution of an elastic variable $\mathsf{u}$ and of an internal variable $\mathsf{z}$. In those papers, 
   $\mathsf{z}$ has relaxation time $\eps$,   while 
   $\mathsf{u}$ has a viscous damping with relaxation time $\eps^\alpha$, $\alpha >0$. 
To emphasize the occurrence  of these three time scales (the time scale $\eps^0=1$ of the external loading,  the relaxation time $\eps$ of $\mathsf{z}$, and the - possibly different - relaxation time $\eps^\alpha$ of $\mathsf{u}$), 
the term `multi-rate system' was used in  \cite{MRS14}. Therein, as well as in 
   \cite{MieRosBVMR}, it was shown that, in  the
three cases $\alpha \in (0,1)$, $\alpha=1$ and $\alpha>1$, 
   the vanishing-viscosity analysis   leads  to different notions of 
 Balanced-Viscosity solutions,   including in particular different descriptions of the
jump behavior developing in the limit $\eps \down 0$. 
   \par
   Thus, along the lines of       \cite{MRS14}, in \cite{Crismale-Rossi} the authors addressed the following, alternative, viscous regularization of system \eqref{RIS-intro}:
   \begin{subequations}
 \label{RD-intro}
\begin{align}
\label{VISC-mom-balance-intro}
& 
- \mathrm{div}(\eps\nu  \mathbb{D}   \rmE{(\dot{u})}   + \sigma ) = f    &&   \text{ in } \Omega \times (0,T), 
\\
& 
\label{VISC-dam-intro}
\partial\did{\dot{z}} +\eps \dot{z} + \As (z) + W'(z) \ni - \tfrac12\bbC'(z)e : e   &&   \text{ in } \Omega \times (0,T),
\\
\label{VISC-pl-intro}
& \partial_{\dot{p}}  \dip{z}{\dot{p}} + \eps\nu  \dot{p} +\newmu  p \ni   \sigma_{\mathrm{D}}   &&  \text{ in } \Omega \times (0,T),
 \end{align}
{supplemented by the boundary conditions}
\begin{equation}
 u=w  \text{ on } \Gdir \times (0,T),  \quad  (\eps\nu  \mathbb{D}  \rmE{(\dot{u})}  + \sigma) \rmn = g \text{  on }\Gneu \times (0,T),  \quad   \partial_\rmn  z =0 \text{ on } \partial\Omega  \times (0,T),
 \end{equation}
  \end{subequations} 
 where $\mathbb{D}$ is a positive-definite fourth-order tensor.
System \eqref{RD-intro} features  a viscous regularization \emph{both} in
the damage flow rule \emph{and}  in 
 the displacement equation  and  the  plastic flow rule. 
Let us now illustrate the role of the various parameters appearing therein, namely
  the
  \begin{itemize}
  \item[-] (vanishing-)viscosity parameter $\eps>0$,
  \item[-] (vanishing-)hardening parameter $\mu>0$,
  \item[-] additional parameter $\nu>0$, which was required to fulfill $ \nu \leq \mu $  in order to get suitable a priori estimates. We have referred to $\nu$ as a \emph{rate} parameter, since it sets
  the mutual rate at which, on the one hand, the   displacement and the plastic strain converge to equilibrium and rate-independent
  evolution, and,  on the other hand,  the  damage parameter
  converges to rate-independent evolution. More precisely, 
 if $\nu>0$ stays fixed 
then
$u$ and $p$ converge at the same rate as $z$, while their  convergence occurs at a faster rate if $\nu \down0$. This is clear if one chooses  e.g.\  $\nu = \mu = \eps$, so that the viscous terms $  \rmE{(\dot{u})}$ and $\dot p $ in \eqref{VISC-mom-balance-intro} and 
 \eqref{VISC-pl-intro} are modulated by 
the 
  coefficient $\eps^2 $, as opposed to the coefficient $\eps$ in the damage flow rule.
   Observe that, upon taking the vanishing-hardening limit $\mu \down 0$,  the constraint $\nu \leq \mu$ forces the 
  joint
  vanishing-viscosity and vanishing-hardening limit to occur at a faster rate for $u$ and $p$ than for $z$. 
  \end{itemize}
 \par We will refer to the vanishing-viscosity analyses in system \eqref{RD-intro} as \emph{full}, as opposed to the \emph{partial} vanishing-viscosity approximation provided by system \eqref{viscous-CL-intro}, where only the damage flow rule is regularized. 
 \par
  Indeed, in \cite{Crismale-Rossi} three \emph{full} vanishing-viscosity analyses
   have been carried out for system \eqref{RD-intro}, 
    leading to three different notions of solution for system \eqref{RIS-intro}, possibly regularized by a hardening term. Let us briefly illustrate them. 
   \begin{description}
   \item[\textbf{(1) $\BVB{\mu,\nu}$ solutions}]   The limit passage in a (reparameterized) version of  \eqref{RD-intro} as $\eps \down 0$, while the positive parameters $\mu$ and $\nu$ stayed fixed,  has led to  $\BV$ solutions for a variant of the  system \eqref{RIS-intro}, where the plastic flow rule was regularized by the hardening term $\mu p$,
   i.e.\ for 
   the rate-independent system with hardening
   \begin{subequations}
   \label{RIS-hard-intro}
   \begin{align}
    & - \mathrm{div} ( \bbC(z) e) = f    &&    \text{ in } \Omega \times (0,T), 
    \\
&  \partial\did{\dot{z}}   + \As (z)+ W'(z) \ni - \tfrac12\bbC'(z)e : e  &&   \text{ in } \Omega \times (0,T),
\\
& \partial_{\dot{p}}  \dip{z}{\dot{p}}   + \newmu p  \ni \sigma_{\mathrm{D}}  &&   \text{ in } \Omega \times (0,T), 
\end{align}
{coupled with the boundary conditions}
\begin{equation}
u=w  \text{ on } \Gdir \times (0,T),   \quad 
 \sigma \rmn=g \text{ on }  \Gneu \times (0,T)\quad \qquad \partial_\rmn z =0 \text{ on } \partial\Omega \times (0,T). 
\end{equation}
   \end{subequations}
    The $\BV$ solutions to system \eqref{RIS-hard-intro} thus constructed reflect their origin from a  system viscously regularized in all of the three variables $u$, $z$, $p$. In fact,   in the jump regime  the system may switch to viscous behavior in the three variables  $u$, $z$, and  $p$.  Since  the convergence of $u$, $z$, and $p$ to 
   elastic equilibrium and rate-independent evolution 
   has occurred  at the \emph{same rate} (as $\nu>0$ stayed fixed), viscous behavior in 
   $u$, $z$, and $p$ may \emph{equally} intervene in the jump regime.  We shall  refer to  such solutions as $\BVB{\mu,\nu}$
   \emph{solutions to system}{} \eqref{RIS-hard-intro}. The  subscript  $0$ suggests that they have been obtained in the vanishing-viscosity limit $\eps\down 0$, while the occurrence of the parameter $\mu$ keeps track of the 
   presence of hardening. Also the parameter $\nu$ appears in the notation,
   since
   it  still features in the limiting evolution as a coefficient 
of the viscous terms in the displacement equation and in the plastic flow rule, which
   may be active in the jump regime. 
   \item[\textbf{(2) $\BVA{\mu}$ solutions}] The limit passage  in a (reparameterized) version of  \eqref{RD-intro} as $\eps \down 0$ simultaneously with $\nu \down 0$,  while $\mu>0$  stayed fixed,  has again led to  $\BV$ solutions for  the rate-independent  elasto-plastic damage system with hardening \eqref{RIS-hard-intro}. 
   These solutions still  have   the feature that, 
 in the jump regime,  the system may   switch to viscous behavior in   $u$, $z$, and  $p$.  However, the $\BV$ solutions thus obtained reflect the fact   that  the convergence of $u$ and $p$ to 
   elastic equilibrium and rate-independent evolution 
  has    occurred  at a \emph{faster rate} (as $\nu\down 0$) than that for $z$. To emphasize this,  such solutions were termed  
  \emph{$\BV $ solutions to the multi-rate system for damage with hardening}. 
   We will refer to them  as \emph{$\BVA{\mu}$ solutions to system \eqref{RIS-intro}}. In this notation, the double occurrence of $0$ 
 relates to the fact that such solutions were obtained in  the limit $\eps,  \nu \down 0$, as opposed to the $\BVZ$ solutions from \cite{Crismale-Lazzaroni} (arising  in the limit of system \eqref{RD-intro} as $\eps \down 0$ and $\mu=\nu=0$). 
  \item[\textbf{(3) $\BVZZZ$ solutions}]   The limit passage in a (reparameterized) version of  \eqref{RD-intro} as $\eps, \, \mu, \, \nu \down 0$
  \emph{jointly}
   led to  \emph{$\BV $ solutions to the multi-rate system for damage and perfect plasticity} \eqref{RIS-intro}, again reflecting the fact that the convergence of $u$ and $p$ to 
   elastic equilibrium and rate-independent, \emph{perfectly plastic} evolution happened at a rate faster than that for $z$.  The vanishing-viscosity solutions arising from this joint limit will receive specific attention in this paper. In what follows, we will refer to them  as \emph{$\BVZZZ$ solutions to system \eqref{RIS-intro}}.
   Here, the triple occurrence of $0$ 
 relates to the fact that such solutions were obtained in  the limit $\eps, \, \mu, \, \nu \down 0$  
 and thus immediately suggests the comparison with 
 the 
 $\BVZ$ solutions from~\cite{Crismale-Lazzaroni}. 
   \end{description} 
 \subsection*{Our results}
 The aim of this paper is twofold:
 \begin{itemize}
 \item[(i)] We propose to  gain further insight into 
   $\BVZZZ$ solutions to system \eqref{RIS-intro} 
 (cf.\ item $\#3$ in the above list).  More precisely,  first of all  we shall provide a \emph{differential characterization} of such solutions, cf. {\bf \underline{Proposition 
 \ref{prop:diff-charact-BV-CR}}} ahead. 
  Relying on that,  in 
  {\bf \underline{Theorem
 \ref{thm:delusione}}}
  we will subsequently prove that, 
  after an initial phase in which  $z$
is constant  while $u$ and $p$, evolving by viscosity,   relax to elastic equilibrium  and to rate-independent evolution, respectively, 
it turns out that
$u$ never leaves the equilibrium, and $p$ the rate-independent regime. 
Afterwards,
the evolution of  system  \eqref{RIS-intro}  is captured by the notion of 
 $\BVZ$ solution as obtained in \cite{Crismale-Lazzaroni} by taking the vanishing-viscosity limit as $\eps\down 0  $ of system \eqref{viscous-CL-intro}. In other words, viscosity in $u$ and $p$ (may) intervene only in an initial phase in the reparameterized time scale, corresponding to a time discontinuity in the time scale of the loading.  After this initial phase, 
  the $\BVZZZ$  solutions  to the perfectly plastic system for damage  \eqref{RIS-intro}
 arising by the \emph{full} vanishing-viscosity approach of   \cite{Crismale-Rossi}
   comply with  the same notion of solution of 
    \cite{Crismale-Lazzaroni}, where viscosity for $u$ and $p$ was neglected.
 \par
We point out  that   an analogous  characterization
can be proved for the  $\BVA{\mu}$ solutions to the multi-rate system with  \emph{fixed} hardening parameter $\mu>0$,   obtained in the limit passage  $\#2$ of the above list; see Remark \ref{rmk:comment-below} ahead. 
  \item[(ii)] 
  We aim to `close the circle' in the analysis of the singular limits of system \eqref{RD-intro}, by showing that,
  for two given sequences $(\mu_k)_k,\, (\nu_k)_k \subset (0,+\infty)$ with $ 0 < \nu_k \leq \mu_k \down 0$  as $k\to\infty$, 
  \begin{enumerate}
  \item $\BVB{\mu_k,\nu_k}$ solutions to the \emph{single}-rate system with hardening 
  converge  as $k\to\infty$   to a $\BVZZZ$ solution of the perfectly plastic damage system, which will be 
  shown in Theorem \ref{teo:exparBVsolsingle} ahead;
  \item $\BVA{\mu_k}$ solutions to the \emph{multi}-rate system with hardening converge  as $k\to \infty$  to a  $\BVZZZ$ solution, cf.\ Theorem \ref{teo:exparBVsol}. 
  \end{enumerate}
  In particular, we will prove that the diagram  in Figure \ref{fig:diag}  commutes. 
  \end{itemize}
%
\begin{figure}[!h]
\begin{center}
\begin{tikzpicture} 
  \matrix (m) [matrix of math nodes,row sep=6em,column sep=8em,minimum width=4em]
	{
				&	\begin{array}{c} \BVB{\mu,\nu} \\ \mu\ge\nu>0 \end{array}	
\\
	\begin{array}{c} \mathrm{V}_\eps^{\mu,\nu} \\ \eps,\nu,\mu>0 \end{array}	&	\BVZZZ 							&	
\\
				&	\begin{array}{c} \BVA{\mu} \\ \mu>0 \end{array}	&	
\\
	};
  \path[-stealth]
	(m-2-1) edge node [above] {\tiny \hspace{1em} \begin{turn}{34} \hspace{-1em} $\eps\to0$\end{turn}} node [below,right] {\tiny \hspace{-1em} \begin{turn}{34} \hspace{-1em} [CR19,\S6.1] \end{turn}} (m-1-2)
	(m-2-1) edge node [above] {\tiny $\eps,\nu,\mu\to0$} node [below] {\tiny [CR19,\S7]} (m-2-2)
	(m-2-1) edge node [below] {\tiny \hspace{1em} \begin{turn}{-34} \hspace{-2em}  [CR19,\S6.2]\end{turn}} node [right] {\tiny \hspace{-1em} \begin{turn}{-34} \hspace{-1em} $\eps,\nu\to0$\end{turn}} (m-3-2)
	(m-1-2) edge [dashed] node [left] {\tiny \begin{turn}{90}$\nu\le\mu\to0$\end{turn}} node [right] {\tiny \begin{turn}{90}(Theorem \ref{teo:exparBVsolsingle})\end{turn}} (m-2-2)
	(m-3-2) edge [dashed] node [left] {\tiny \begin{turn}{90}$\mu\to0$\end{turn}} node [right] {\tiny \begin{turn}{90}(Theorem \ref{teo:exparBVsol})\end{turn}} (m-2-2)
%
;
\end{tikzpicture}
\caption{\small The diagram displays the asymptotic relations   between different notions of solution.
The symbol $\mathrm{V}_\eps^{\mu,\nu}$ indicates solutions to the viscous  system \eqref{RD-intro}.
Solid lines represent convergences (along sequences) to the limiting solutions of type $\BVB{\mu,\nu}$, $\BVA{\mu}$, $\BVZZZ$,
already proved in  \cite[Sections 6.1, 6.2, 7]{Crismale-Rossi}.
Dashed lines represent convergences proved in the present paper, the corresponding theorems being referred to  in the diagram.
Starting from $\mathrm{V}_\eps^{\mu,\nu}$ one may either pass to the limit as $\eps\down0$ and then as $\mu,\,\nu\down0$; or pass to the limit as $\eps,\,\nu\down0$ and then as $\mu\down0$.
Since there is no uniqueness, it is not guaranteed that one gets the very same solution found in the joint limit $\eps, \, \mu, \, \nu \down 0$.
However, we prove that through the three different procedures one finds evolutions satisfying the same \emph{notion} of solution.
In this sense we may  say that the diagram commutes. 
}
\label{fig:diag}
\end{center}
\end{figure}
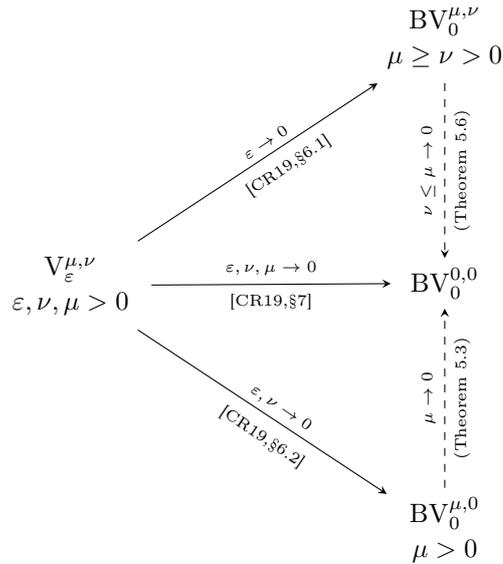

We emphasize that these 
 results establish asymptotic relations between  $\BV$ solutions, already obtained as  vanishing-viscosity limits. These convergence analyses
show that $\BVZZZ$ solutions are robust enough to capture the asymptotic behavior of a wide class of  $\BV$  solutions depending on different parameters. However, $\BVZZZ$ solutions reduce to $\BVZ$ solutions  after an initial phase in which $u$ and $p$ converge to elastic equilibrium and stability, respectively; in particular, if the initial conditions are at equilibrium, then the two notions of solution coincide. 
This feature may be traced back to the convex character of perfect plasticity and to the multi-rate 
character inherent to the system, since  with $\nu \down 0$ we have forced faster convergence to equilibrium in $u$ and stability in $p$. 
  \paragraph{\bf Plan of the paper.} In  Section \ref{s:2} we  detail  the setup of the problem, 
  list our assumptions, and provide some preliminary results.
 In  Section \ref{s:3}  we illustrate  the notion of $\BVZ$ solution to system 
  \eqref{RIS-intro} arising from the \emph{partial}  vanishing-viscosity  approach of \cite{Crismale-Lazzaroni}, and  
 that of $\BVZZZ$ solution via the \emph{full} vanishing-viscosity analysis in  
   \cite{Crismale-Rossi}.
 In Section~\ref{sez:complete-charact}  we establish Theorem  \ref{thm:delusione}, providing a complete characterization of $\BVZZZ$ solutions. Section~\ref{s:4} is devoted  to the vanishing-hardening analysis of 
 $\BVA{\mu}$ and $\BVB{\mu,\nu}$ solutions to the system with hardening. The proofs of Theorems  \ref{teo:exparBVsol} and \ref{teo:exparBVsolsingle} rely on some technical results collected in the Appendix.


\section{Setup and  preliminaries}
\label{s:2}
Throughout the paper we will use the following 
\begin{notation}[General notation  and preliminaries] 
\label{not:1.1}
\upshape
Let $X$ be a Banach space. By
$\pairing{}{X}{\cdot}{\cdot}$  we denote the duality between $X^*$ and $X$ or between $(X^n)^*$ and $X^n$
(whenever $X$ is a Hilbert space, $\pairing{}{X}{\cdot}{\cdot}$  will be the inner product),  while $\| \cdot \|_{X}$ stands for  the norm in $X$ or in $X^n$.
The inner Euclidean product in $\R^n$, $n\geq 1$, is denoted by $\pairing{}{}{\cdot}{\cdot}$ and the Euclidean norm in $\R^n$  by $|\cdot| $.
The symbol
 $B_r(0)$  stands for the open ball in  $\R^n$ with radius $r$ and center $0$. 
\par
We  write $\| \cdot \|_{L^p}$ for the $L^p$-norm on the space $L^p(O;\R^d)$, with $O$ a measurable subset of $\R^n$ and $1\leq p<+\infty$,  and similarly $\| \cdot \|_{\Hs}$ for the norm of the Sobolev-Slobodeskij space $\Hs(O)$,  for $0\le{\rm m}\in\R$.  The symbol $\Mb(O;\R^d)$ stands for the space of $\R^d$-valued bounded Radon measures in $O$.
\par
The space of symmetric  $(n{\times}n)$-matrices is denoted by $\Mnn$, while the subspace of the deviatoric matrices   with null trace is denoted by $\MD$. One has
$\Mnn = \MD \oplus \R I$, where $I$ is the identity matrix, i.e.\ any $\eta \in \Mnn$ can be decomposed as 
$
\eta = \eta_\dev+ \frac{\mathrm{tr}(\eta)}n I
$,
where $\eta_\dev$ is the orthogonal projection of $\eta$ onto $\MD$. The latter is called the \emph{deviatoric part}  of $\eta$. 
The symbol $\Sym(\MD;\MD)$ stands for the set of symmetric endomorphisms on $\MD$. 
 \par
Given a function $v:\Omega\times (0,T)\to\R$ differentiable,  w.r.t.\ time a.e.\ on $\Omega\times (0,T)$, its (almost everywhere defined) partial time derivative is indicated by $\dot{v}:\Omega\times (0,T)\to\R$. 
A different notation will be employed when considering $v$ as a (Bochner) function,  from $(0,T)$ with values in a Lebesgue or Sobolev space $X$  (with the Radon-Nikod\'ym property):
if $v\in\AC ([0,T];X)$,  then its  (almost everywhere defined)  time derivative is indicated by $v':(0,T) \to X$. 
\par
The symbols $c,\,c',\, C,\,C'$ will denote positive constants whose precise value may vary from line to line (or within the same   line).
We will sometimes employ the symbols $I_i$,  $i = 0, 1,... $,
 as place-holders for  terms appearing in inequalities: also in this case, such symbols may appear in different proofs with different meaning.   
\end{notation}
\subsubsection*{Function of bounded deformation}
The state space for the displacement variable for the systems with hardening will be
\[
H_{\Dir}^1(\Omega;\R^n): = \{ u \in H^1(\Omega;\R^n)\, : \ u =0 \text{ on } \Gamma_\Dir\}
\]
(recall that $ \Gamma_\Dir$ is the Dirichlet part of $\partial\Omega$, cf.\ \eqref{Omega-s2} ahead). 

For the perfectly plastic damage system,  displacements will   belong to the space of   \emph{functions of bounded deformations},   defined by
\begin{equation*}
\BD(\Omega): = \{ u \in L^1(\Omega;\R^n)\, : \ \sig{u} \in  \Mb(\Omega;\Mnn)  \},
\end{equation*}
 with  $\Mb(\Omega;\Mnn) $    the space of  $\Mnn$-valued bounded Radon measures on $\Omega$.
We recall that 
 $\Mb(\Omega;\Mnn) $ can be identified with the dual of the space $\mathrm{C}^0_0(\Omega;\Mnn )$   of continuous $\Mnn$-valued functions vanishing at the boundary of $\Omega$.  
The space $\BD(\Omega)$ has a Banach structure if equipped with the norm 
\[
\| u \|_{\BD(\Omega)}: = \| u \|_{L^1(\Omega;\R^n)}+ \| \sig{u}\|_{\Mb(\Omega;\Mnn) }.
\]
Indeed,  $\BD(\Omega)$ is the dual of a normed space, cf.\ \cite{Temam-Strang80}, and 
 such duality provides a  weak$^*$ convergence on $\BD(\Omega)$: 
a sequence $(u_k)_k $ converges to $u$  weakly$^*$   in $\BD(\Omega)$ if $u_k\weakto u$ in $L^1(\Omega;\R^n)$ and $\sig{u_k}\weaksto \sig{u}$ in $ \Mb(\Omega;\Mnn)$.  
 It holds $\BD(\Omega)\subset L^{n/(n{-}1)}(\Omega; \R^n)$.  
Up to subsequences, every bounded sequence in $\BD(\Omega)$ converges weakly$^*$, weakly in $L^{n/(n{-}1)}(\Omega;\R^n)$, and strongly in $L^{p}(\Omega;\R^n)$  for any $1\leq p <\frac n{n-1}$. 
Finally, we recall that the trace $u|_{\partial\Omega}$ of a function $u \in \BD(\Omega)$  is well defined and is an element in $L^1(\partial\Omega;\R^n)$. 
\subsubsection*{A divergence operator}
First of all, we observe that  
 any $\sigma \in  \Lnn $ such that $\mathrm{div}\, \sigma  \in L^2(\Omega;\R^n) $ induces
 the distribution $[\sigma \rmn]$  defined by 
\begin{equation}\label{2809192054}
\langle [\sigma \rmn] , \psi \rangle_{\partial\Omega}:= \langle \mathrm{div}\,\sigma , \psi \rangle_{L^2} + \langle \sigma , \rmE(\psi) \rangle_{L^2}  \qquad \text{for every }  \psi \in H^1(\Omega;\R^n).
\end{equation}
 By \cite[Theorem~1.2]{Kohn-Temam83} and \cite[(2.24)]{DMDSMo06QEPL} we have that $[\sigma \rmn] \in H^{-1/2}(\partial \Omega; \R^n)$; 
moreover, if $\sigma \in \mathrm{C}^0(\overline{\Omega};\Mnn)$, then  the distribution $[\sigma \rmn]$ fulfills $[\sigma \rmn] = \sigma \rmn$,
where the right-hand side is 
the standard pointwise product of the matrix $\sigma$ and the normal vector $\rmn$ in $\partial \Omega$. 
\par
 For the treatment of the perfectly plastic system for damage it will be crucial to   work with the space 
\begin{equation}
\label{space-sigma-omega}
\Sigma(\Omega):= \{ \sigma \in \Lnn \colon  \mathrm{div}\,\sigma  \in L^n(\Omega;\R^n) , \ \sigma_\dev \in L^{\infty}(\Omega;\MD) \}. 
\end{equation} 
Furthermore, 
our choice of external loadings
(see \eqref{forces-u}) will ensure that 
the stress fields $\sigma$  that we consider, at equilibrium, have the additional property that $[\sigma \rmn] \in L^\infty(\Omega;\R^n)$ and $\sigma \in \Sigma(\Omega)$ (cf.\ Lemma~\ref{le:0103222042}). Therefore,  any of such fields
induces a functional $-\Diver\,\sigma \in \BD(\Omega)^*$ via 
\begin{equation}\label{2307191723}
\langle -\Diver\,\sigma, v \rangle_{\BD(\Omega)} :=\langle - \mathrm{div}\,\sigma,   v \rangle_{L^{\frac{n}{n-1}}(\Omega; \R^n )}  + \langle [\sigma \rmn], v \rangle_ {L^{1}(\Gneu; \R^n)}  \qquad \text{for all } v \in \BD(\Omega). 
\end{equation}
With slight abuse of notation, we shall denote by $ -\Diver\,\sigma$ also the restriction of the above functional to $H^1(\Omega;\R^n)^*$. 
%
\subsubsection*{The $\As$-Laplacian}
The damage flow rule features 
   a gradient regularizing contribution
   in terms of the $\As$-Laplacian operator, that is defined from  the bilinear form
\begin{align*}
\label{bilinear-s_a}
\ass : \Hs(\Omega) \times\Hs(\Omega) \to \R, \quad 
\ass(z_1,z_2): =
 \int_\Omega \int_\Omega\frac{\big(\nabla z_1(x) - \nabla
   z_1(y)\big)\cdot \big(\nabla z_2(x) - \nabla
   z_2(y)\big)}{|x-y|^{n  + 2 (\mathrm{m} - 1)}}\dd  x \dd y \text{ with } \mathrm{m} > \frac n2.
\end{align*}
Then, 
\begin{equation*}
\label{m-Laplacian}
\As: \Hs(\Omega) \to \Hs(\Omega)^* \text{ is defined by } 
 \langle \As(z), w \rangle_{\Hs(\Omega)}  := \ass(z,w) \quad \text{for
every $z,\,w \in \Hs(\Omega)$.}
\end{equation*}
 The inner product $\pairing{}{\Hs(\Omega)}{z_1}{z_2}: = \int_\Omega z_1 z_2 \dd x + \ass(z_1,z_2)$ makes $\Hs(\Omega)$ a Hilbert space. 
Throughout the paper we shall assume $\mathrm{m}>\tfrac n2$ and rely on  the compact embedding $\Hs(\Omega)\Subset \mathrm{C}^0(\overline\Omega)$. 
\subsection{Assumptions and preliminary results}
\label{ss:2.1}
This section and Sec.\ \ref{ss:2.2}  collect all our assumptions on the constitutive functions of the model and on the problem data. We will omit to invoke them explicitly in the statement of the various results. 
\subsubsection*{The reference configuration}
In what follows we will assume that $\Omega \subset \R^n$,  $n\in \{2,3\}$,   is a  bounded  Lipschitz  domain
 satisfying 
 the so-called \emph{Kohn-Temam condition} 
\begin{equation}
\label{Omega-s2}
\tag{2.$\Omega$}
\begin{gathered}
\partial \Omega = \Gdir \cup
\Gneu \cup \Sigma \quad \text{ with $\Gdir, \,\Gneu, \, \Sigma$ pairwise disjoint,}
\\
\text{
 $\Gdir$ and $\Gneu$ relatively open in $\partial\Omega$, and $ \partial\Gdir = \partial \Gneu = \Sigma$ their relative boundary in $\partial\Omega$,}
  \\
  \text{ with $\Sigma$ of class $\mathrm{C}^2$ and $\calH^{n-1}(\Sigma)=0$, and with $\partial\Omega$ Lipschitz and of class $\mathrm{C}^2$ in a neighborhood of $\Sigma$.}
  \end{gathered}
  \end{equation} 
\subsubsection*{The elasticity and viscosity tensors}
We assume that the elastic tensor $\C : [0,+\infty) \to   \Lin(\Mnn;\Mnn)  $
 fulfills  the following conditions
\begin{align}
&
\tag{$2.\mathbb{C}_1$}
 \C \in \mathrm{C}^{1,1}([0,+\infty);  \Lin(\Mnn;\Mnn)),
 \label{spd}
\\
&
\tag{$2.\mathbb{C}_2$}
 z  \mapsto \C(z) \xi : \xi \ \text{is nondecreasing for every}\ \xi \in \Mnn,   \label{C3} \\ 
&
\tag{$2.\mathbb{C}_3$}
\exists\, \gamma_1,\, \gamma_2>0  \ \ \forall\, z \in [0,+\infty)  \ \forall\, \xi \in \Mnn \, : \quad  \gamma_1 |\xi |^2  \leq \C(z) \xi : \xi \leq \gamma_2 |\xi |^2  . \label{C2}
\end{align}
For the  viscosity tensor $\mathbb{D}$ we require that 
\begin{align}
\label{visc-tensors-1} 
\tag{$2.\bbD_1$}
&
  \bbD
  \in \mathrm{C}^0(\overline{\Omega};  \Sym(\MD;\MD)),  \text{ and } 
  \\
&
\label{visc-tensors-2} 
\tag{$2.\bbD_2$}
 \ \exists\,  \delta_1,\, \delta_2>0  \  \forall x \in \Omega \   \ \forall A \in \mt_\sym^{d\times d} \, : \quad \ \delta_1 |A|^2 \leq  \bbD(x) A : A \leq  \delta_2 |A|^2 .
  \end{align}
Thus, $\bbD$ induces an equivalent (by a Korn-Poincar\'e-type inequality)
   Hilbert norm on $H_\Dir^1(\Omega;\R^n)$, i.e.\
\begin{equation}
\label{norma-equivalente}
\| u \|_{H^1, \bbD} : = \left(  \int_{\Omega} \bbD \sig{u} : \sig{u} \dd x \right)^{1/2} \quad \text{ and } \quad
 \exists\, K_{\bbD} >0 \ \forall\,  u \in  H_\Dir^1(\Omega;\R^n) : \ 
 \| u \|_{H^1, \bbD} \leq K_{\bbD} \| \sig{u}\|_{L^2}\,.
\end{equation}
The related 
 `dual norm' is
\begin{equation}
\label{norma-duale}
\| \eta \|_{(H^1, \bbD)^*}: = \left( \int_\Omega \bbD^{-1} \xi : \xi \right)^{1/2}  \quad\text{for all } \eta \in H_\Dir^1(\Omega;\R^n)^* \text{ with } \eta = \mathrm{Div}\,\xi  \text{ for some } \xi \in \widetilde{\Sigma}(\Omega). 
\end{equation} 
  \subsubsection*{The potential energy for the damage variable} In addition to  the regularizing, nonlocal gradient contribution featuring the bilinear form $\ass$, the $z$-dependent part of the mechanical energy functional shall feature a further term with density $W$ satisfying 
\begin{align}
&
\tag{$2.W_1$}
W\in \rmC^{2}((0,+\infty); \R^+)\cap  \rmC^0([0,+\infty); \R^+{\cup}\{+\infty\}),\label{D1}\\
&
\tag{$2.W_2$}
s^{2n} W(s)\rightarrow +\infty \text{ as }s\rightarrow 0^+, \label{D2}
\end{align}
where $W\in \rmC^0([0,+\infty); \R^+\cup\{+\infty\})$ means that  $W(0)=\infty$ and
$W(z)\to +\infty$ if $z\down 0$, in accordance with 
\eqref{D2}. 

Indeed,  the energy contribution involving  $W$ forces  $z$ to be strictly  positive;  consequently, the material never reaches the most damaged state at any point.
\subsubsection*{The plastic and the damage dissipation densities}
The plastic dissipation  potential shall reflect the  requirement that the admissible stresses belong to given constraint sets which,   in turn, 
depend on the damage variable $z$.
More precisely, 
as in 
\cite{Crismale-Lazzaroni} we  ask  that the constraint sets $(K(z))_{z\in [0,+\infty)}$   fulfill
\begin{align}
&
\label{Ksets-1}
\tag{$2.K_1$}
K(z) \subset \MD \text{ is closed and convex for all }  z\in [0,+\infty),
\\
&
\label{Ksets-2}
\tag{$2.K_2$}
\exists\,  0<\bar{r} <  \bar{R} \quad   \forall\, 0\leq z_1\leq z_2  \, : \qquad B_{\bar r} (0) \subset K(z_1)\subset K(z_2) \subset   B_{\bar{R}} (0),  
\\
&
\label{Ksets-3}
\tag{$2.K_3$}
\exists\, C_K>0 \quad \forall\, z_1,\, z_2 \in   [0,+\infty)\, :   \qquad d_{\mathscr{H}} (K(z_1),K(z_2)) \leq C_K  |z_1{-}z_2|, 
\end{align}
with 
$d_{\mathscr{H}}$ the Hausdorff distance between two subsets of $\MD$, defined by
\[
 d_{\mathscr{H}} (K_1,K_2): =  \max \left( \sup_{x\in K_1} \mathrm{dist}(x,K_2), \, \sup_{x\in K_2} \mathrm{dist}(x,K_1) \right).
\]
The associated  support function  $H:[0,+\infty)  \times  \MD \to [0,+\infty)$,  defined by
\begin{equation}\label{0307191050}
H(z,\pi): = \sup_{\sigma \in K(z)} \sigma : \pi \qquad \text{for all } (z,\pi) \in [0,+\infty)\times  \MD ,
\end{equation}
will act as density function for the plastic dissipation potential, cf.\ \eqref{dam-diss-pot} later on.
\par
  We choose as  damage dissipation density the function $ \mathrm{R} : \R \to [0,+\infty] $ given by 
 \begin{equation}
 \label{damage-dissipation-density}
 \mathrm{R}(\zeta): = \begin{cases}
 -\kappa \zeta & \text{if } \zeta \leq 0,
 \\
 +\infty & \text{otherwise,}
 \end{cases}
 \end{equation}
with $\kappa>0$ a constant related to the toughness of the material.
\subsubsection*{Body and surface forces, Dirichlet loading, and initial data}
We assume that   the volume force $f$ and the assigned traction $g$ fulfill  
 \begin{subequations} \label{hyp-data}
  \begin{equation}
\label{forces-u}
f \in H^1(0,T;  L^n(\Omega;\R^n) ) , \quad g \in H^1(0,T; L^\infty(\Gneu;\R^n)).
\end{equation}
 The induced  total load is the function 
\[
F\colon [0,T] \to  \BD(\Omega)^*, \qquad 
\langle F(t), v \rangle_{\BD(\Omega)}: = \langle f(t), v\rangle_{ L^{n/(n{-}1)}(\Omega;\R^n)} + \langle g (t), v \rangle_{L^1(\Gamma_\Neu;\R^n)}
\quad \text{for all } v \in \BD(\Omega).
\]
Furthermore, as customary for perfect plasticity, 
 we  shall impose  a \emph{uniform safe load condition}, namely that  there exists 
\begin{equation}\label{2909191106}
\rho \in H^1(0,T; \Lnn) \quad\text{ with }\quad\rho_\dev \in H^1 (0,T;\Linftyn)
\end{equation} and there exists $\alpha>0$ such that  for every $t\in[0,T]$  (recall \eqref{2809192054})
\begin{equation}\label{2809192200}
-\mathrm{div}\,\varrho(t)=f(t) \text{ a.e.\ on }\Omega,\quad\qquad [\varrho(t) \rmn]= g(t) \text{ on } \Gneu,
\end{equation}
\begin{equation}
\label{safe-load}
\rho_\dev(t,x) +\xi \in K  \qquad \text{for a.a.}\ x\in\Omega \ \text{and for every}\ \xi\in\Mnn \ \text{s.t.}\ |\xi|\le\alpha.
\end{equation} 
Observe that, combining \eqref{forces-u} with     \eqref{2909191106}--\eqref{safe-load}  yields
$-\Diver\,\varrho(t)=F(t)$ for all $t \in [0,T]$. 
\par
As for the time-dependent Dirichlet boundary condition  $w$, we  assume  that
\begin{equation}
\label{dir-load} 
 w \in H^1(0,T;H^1(\R^n;\R^n)).  
\end{equation}
Finally,  we shall consider initial data  $q_0=(u_0, z_0, p_0)$ with 
\begin{equation}
\label{init-data} 
u_0 \in  H_\Dir^1(\Omega;\R^n),  \qquad z_0 \in \Hs(\Omega) \text{ with } W(z_0) \in L^1(\Omega) \text{ and } z_0 \leq 1 \text{ in } \ol\Omega,  \qquad p_0 \in L^2(\Omega;\MD). 
\end{equation}
\end{subequations}
\subsubsection*{The stress-strain duality} For the treatment of the perfectly plastic damage system it is essential to resort to
 a suitable notion of stress-strain duality that we borrow from    \cite{Kohn-Temam83,DMDSMo06QEPL}, also relying on  the more recent extension to Lipschitz boundaries from  \cite{FraGia2012},
to which we refer for the properties mentioned below.
 Following \cite{DMDSMo06QEPL} we introduce 
the class  $A(w)$ of 
\emph{admissible displacements and strains}
 associated with a function   $w \in H^1(\R^n; \R^n)$, 
  that is
\begin{equation*}
\begin{split}
A(w):=\{(u,e,p) \in  \, & \BD(\Omega) \times L^2(\Omega;\Mnn) \times  \MbD  \colon \\
& \rmE(u)   =e+p \text{ in }\Omega,\, p=(w-u){\,\odot\,}\rmn\,\hn \text{ on }  \Gdir \},
\end{split}
\end{equation*}
 where $ \mathrm{n} $ denotes the normal vector to $\partial \Omega$ and $\odot$ the symmetrized tensorial product. 
 The \emph{space of admissible plastic strains} is 
\begin{equation*}
\begin{split}
\Pi(\Omega) := \{p\in  \MbD  \colon  \exists \, (u, w,e)\in \BD(\Omega)\times   H^1(\R^n;\R^n)  \times L^2(\Omega;\Mnn)\,
 \text{ s.t.}\, (u,e,p)\in A(w) \} .
\end{split}
\end{equation*}
 Given $\sigma \in \Sigma(\Omega)$ (cf.\ \eqref{space-sigma-omega}), $p \in \Pi(\Omega)$, and $u,\,e$ such that $(u,e,p)\in A(w)$ we define 
\begin{equation}\label{sD}
\langle [\sigma_\dev:p], \varphi\rangle:=-\int_\Omega \varphi\sigma\cdot (e{-} \rmE(w)  )\,\mathrm{d}x-\int_\Omega\sigma\cdot[(u{-}w)\odot \nabla \varphi]\,\mathrm{d}x -\int_\Omega \varphi \, ( \mathrm{div}\,\sigma)  \cdot (u{-}w)\,\mathrm{d}x
\end{equation}
for every $\varphi \in \mathrm{C}^\infty_c(\R^n)$;
 in fact, this definition is independent of $u$ and $e$. 
Under these assumptions, 
$\sigma \in L^r(\Omega;\Mnn)$ for every $r < \infty$,
and $[\sigma_\dev:p]$ is a bounded Radon measure  with 
$\|[\sigma_\dev:p]\|_1\leq \|\sigma_\dev\|_{ L^\infty}\|p\|_1$ in $\R^n$.
Restricting such  measure to $ \Omega \cup \Gdir $, we  set 
\begin{equation}
\label{duality-product}
 \langle \sigma_\dev\, |\,p\rangle:=[\sigma_\dev:p]( \Omega \cup \Gdir ).
 \end{equation}
  By \eqref{Omega-s2} and \eqref{sD}, since 
$u\in\BD(\Omega) \subset L^{\frac{n}{n-1}}(\Omega;\R^n)$, 
we get the following integration by parts formula, valid if the distribution $[\sigma \rmn]$ defined in \eqref{2809192054} belongs to
$\in L^\infty(\Gneu;\R^n)$: 
\begin{equation*}
\langle \sigma_\dev\,|\,p\rangle=-\langle\sigma, e- \rmE(w)  \rangle_{\Lnn}+  \langle -\diver\,\sigma, u-w \rangle_{L^{\frac{n}{n-1}}(\Omega; \R^n )}  + \langle [\sigma \rmn], u-w \rangle_{L^{1}(\Gneu; \R^n)}   
\end{equation*}
for every  $\sigma \in \Sigma(\Omega)$ and $(u,e,p)\in A(w)$. 

\subsection{Energetics}
\label{ss:2.2}
A key ingredient for the construction of $\BV$ solutions to the rate-independent systems \eqref{RIS-intro}
 (damage with perfect plasticity)  and \eqref{RIS-hard-intro} (damage and plasticity with hardening) is the observation that
 their 
  rate-dependent regularizations    \eqref{viscous-CL-intro}  and \eqref{RD-intro}
  have a \emph{gradient-system structure}.
 Namely, they  can be reformulated in terms of the generalized gradient flow
 \begin{equation*}
 \label{GS}
 \partial \Psi(q'(t)) + \rmD \calE (t,q(t)) \ni 0 \qquad \text{in } \bfQ^*, \ \aein\, (0,T),
 \end{equation*}
 for  suitable choices of 
 \begin{compactitem}
 \item[-] the state space $\bfQ$ for the triple $q= (u,z,p)$;
 \item[-] the driving energy functional $\calE: (0,T) \times \bfQ \to \R\cup \{+\infty\}$;
 \item[-] the dissipation potential $\Psi: \bfQ \to [0,+\infty]$, with convex analysis subdifferential $\partial\Psi: \bfQ \rightrightarrows \bfQ^*$,
 \end{compactitem}
 as 
  rigorously proved  in \cite{Crismale-Rossi}.
  Observe that also the rate-independent systems \eqref{RIS-intro}
   and \eqref{RIS-hard-intro} have a gradient structure that is, however, only formal
 due to the fact that the functions $u$, $z$, and $p$ may have jumps as functions of time.
 Nonetheless, 
 for our analysis
 it is crucial to 
 detail the energetics underlying both the rate-dependent and the rate-independent systems.
 \subsubsection*{The state spaces}
 The state space for the rate-dependent/independent damage systems with hardening is 
 \begin{equation*}
 \Qha :=  H^1_{\Dir} (\Omega;\R^n)\times \Hs(\Omega) \times L^2(\Omega;\MD).
 \end{equation*}
 For the rate-independent damage system with perfect plasticity, 
 the displacements are just functions of bounded deformation and the plastic strains are only bounded Radon measures on $\Omega \cup \Gamma_\Dir$, so that 
 the associated state space is  
   \begin{equation}
   \label{defQpp}
   \begin{aligned}
   \Qpp: =\{ q=(u,z,p) \in &\,\BD(\Omega)\times \Hs(\Omega) \times \MbD\, : \\& \, e: = \sig{u} -p \in \Lnn, 
   \ u \odot \mathrm{n} \,\hn + p =0 \text{ on }  \Gdir  \} .
   \end{aligned}
   \end{equation}
Observe that in the definition of $\Qpp$ it is in fact required that $(u,e,p) \in A(0)$: indeed,  
the condition $ u \odot \mathrm{n} \, \mathscr{H}^{n-1} + p =0 $  is a relaxation of  the homogeneous Dirichlet condition 
   $u=0 $ on  $\Gdir$. 
\subsubsection*{The  energy functionals} The  energy functional  governing  the rate-dependent and rate-independent systems  \emph{with hardening}  
\eqref{RD-intro}  and \eqref{RIS-hard-intro}, respectively, consists of 
\begin{enumerate}
\item  a contribution featuring the elastic energy 
\[
\calQ (z,e) = \int_\Omega \tfrac12 \bbC(z) e: e \dd x;
\]
\item the potential energy for the damage variable and  for  the hardening term;
\item  the time-dependent volume and surface forces.
\end{enumerate}
 Namely,  for $\mu>0$ given,
 $\enen{\mu} \colon [0,T] \times \Qha \to \R\cup \{ +\infty\}   $ is  defined   by 
  \begin{equation*}
 \begin{aligned}
\enen{\mu} (t,u,z,p): =  \calQ( z,\sig{u{+}w(t)} {-}p)   + \int_\Omega \left(  W(z) {+} \frac{\mu}{2} |p|^2 \right) \,\mathrm{d}x 
+ \frac12 \ass(z,z) - \langle F(t), u+w(t) \rangle_{H^1(\Omega;\R^n)} . 
\end{aligned}
\end{equation*}
In what follows, we will use the short-hand notation
\begin{equation*}
e(t): = \sig{u{+}w(t)} {-}p
\end{equation*}
for the elastic part of the strain tensor. 
\par
The energy functional for the  rate-independent perfectly plastic damage system \eqref{RIS-intro} is 
   $\calE_0: [0,T]\times \Qpp \to  \R{\cup}\{+\infty\}$  given by  
\begin{equation}
\label{ene-plastic}
\begin{aligned}
\enen{0} (t,u,z,p)   :=  \calQ(z,e(t)) + \int_\Omega   W(z) \,\mathrm{d}x + \frac12 \ass(z,z)  - \langle F(t), u+w(t) \rangle_{\BD(\Omega)} .
\end{aligned}
\end{equation}
Observe that 
\begin{equation}
\label{domE0}
\mathrm{dom}(\calE_0) = [0,T]\times \mathrm{D} \qquad\text{with } \mathrm{D}= \{ q \in \Qpp\,: \ W(z)\in L^1(\Omega)\}.
\end{equation}
\subsubsection*{The dissipation potentials}
 The dissipation density $\mathrm{R}$ from \eqref{damage-dissipation-density} clearly  induces an integral functional
 $\calR:L^1(\Omega) \to [0,+\infty]$. However, 
   since the  damage flow rule  will be posed in $\Hs(\Omega)^*$  (cf.\ \eqref{1509172300} ahead),
 we will restrict  $\calR$ to the space $\Hs(\Omega)$,
denoting the restriction  by the same symbol. Hence,  we will work work with the functional
 \begin{equation}
 \label{dam-diss-pot}
 \calR: \Hs(\Omega) \to [0,+\infty], \quad \calR(\zeta) = \int_\Omega \mathrm{R}(\zeta(x)) \dd x
\end{equation}
and with its convex analysis subdifferential $\partial \calR : \Hs(\Omega) \rightrightarrows \Hs(\Omega)^*$. 
 We will often use the following characterization of $\partial \calR$, due to the $1$-homogeneity of the potential $\calR$:
\begin{equation}
\label{charact-calR}
\chi \in \partial \calR(\zeta) \quad \Leftrightarrow \quad \begin{cases} \calR(\tilde \zeta) \geq \pairing{}{\Hs}{\chi}{\tilde \zeta}  \text{ for all }  \tilde \zeta \in \Hs(\Omega)
\\
\calR(\zeta) \leq \pairing{}{\Hs}{\chi}{\zeta} 
\end{cases}
\,.
\end{equation}

 \par
 For the systems with hardening (i.e.\ \eqref{RD-intro} and \eqref{RIS-hard-intro}), the 
  \emph{plastic dissipation  potential} $\mathcal{H} \colon \rmC^0(\overline{\Omega}; [0,+\infty))  \times  L^1(\Omega;\MD) 
\to \mathbb{R}$ is defined by
\begin{equation}
\label{plast-diss-pot-visc}
\mathcal{H}(z,\pi):= \int_{ \Omega} H(z(x),\pi(x))\dd x ,
\end{equation}  
where   $H$ is given by \eqref{0307191050}
and $\pi$ is a place-holder for the plastic rate $\dot p$. 
Its convex analysis subdifferential 
$\partial_\pi \calH \colon \rmC^0(\overline{\Omega}; [0,+\infty)) \times L^1(\Omega;\MD)\rightrightarrows L^\infty(\Omega;\MD)$, w.r.t.\ the  second variable $\pi$, given  by
\[
\omega \in \partial_\pi \calH(z, \pi) \quad \text{ if and only if } \quad  \calH(z,\varrho)- \calH(z,\pi) \geq \int_\Omega \omega (\varrho - \pi) \dd x  \quad \text{for all } \varrho \in L^1(\Omega;\MD),
\]
fulfills
\begin{equation}
\label{subdiff-calH}
\omega \in \partial_\pi \calH(z, \pi) \quad \text{ if and only if } \quad  \omega(x)\in \partial_\pi H(z(x),\pi(x)) \quad \foraa\, x \in \Omega.
\end{equation} 
 A characterization analogous to \eqref{subdiff-calH} holds for  the subdifferential
$\partial_\pi \calH(z,\cdot) \colon  L^1(\Omega;\MD)\rightrightarrows L^\infty(\Omega;\MD)$. 
\par
In order to handle the perfectly plastic system for damage, 
 the plastic dissipation potential $\mathcal{H}(z,\cdot)$ has to be extended to the space
   $\MbD$. 
 We define
 $\Hpp \colon \rmC^0(\overline{\Omega}; [0,1]) \times  \MbD  \to \mathbb{R}$ by
  \begin{equation*}
\Hpp(z,\pi):= \int_{ \Omega \cup \Gdir} H\biggl(z(x),\frac{\mathrm{d}\pi}{\mathrm{d}\mu}(x) \biggr)\,\mathrm{d}\mu(x),
\end{equation*}  
 where  $\mu \in \MbD$ is a positive  measure such that $ \pi \ll \mu $ and $\frac{\mathrm{d} \pi}{\mathrm{d}\mu}$ is the Radon-Nikod\'ym derivative of $p$ with respect to $\mu$; by one-homogeneity of \ $H(z(x),\cdot)$, the definition of  $\Hpp$  does not depend  of $\mu$.
For the theory of convex functions of measures  we refer to \cite{GofSer}. 
By \cite[Proposition~2.37]{AmFuPa05FBVF},    for every   $z \in \rmC^0(\overline{\Omega};[0,1])$  the functional 
 $p \mapsto \Hpp(z,p)$ 
is convex and positively one-homogeneous.
 We recall that by Reshetnyak's lower semicontinuity theorem, if  $(z_k)_k\subset\rmC^0(\overline \Omega;[0,1])$ and $(\pi_k)_k\subset\MbD $ are  such that $z_k \rightarrow z$ in $\rmC^0(\overline\Omega)$ and $\pi_k \rightharpoonup \pi$ weakly$^*$ in $\MbD$, then
 \begin{equation*}
 \Hpp(z,\pi) \leq \liminf_{k \rightarrow +\infty} \Hpp(z_k, \pi_k). \label{Hsci}
 \end{equation*}
 Finally, 
 from \cite[Proposition~3.9]{FraGia2012} it follows that for every $\sigma \in \mathcal{K}_z(\Omega)$ 
\begin{equation}\label{Prop3.9}
H\biggl(z, \frac{\mathrm{d}p}{\mathrm{d} |p|}\biggr)|p| \geq [\sigma_\dev:p] \quad \text{as measures on }  \Omega \cup \Gdir  .
\end{equation}
 In particular, we have 
\begin{equation}\label{eq:carH}
\Hpp(z, p)\geq  \sup_{\sigma  \in  \mathcal{K}_z(\Omega)}  \langle \sigma_\dev\, |\, p\rangle \,\qquad \text{for every $p \in \Pi(\Omega)$.} 
\end{equation}
\par
The (partially) viscously regularized system \eqref{viscous-CL-intro} also features the 
$2$-homogeneous dissipation potential
\[
\calR_2(\zeta): = \int_\Omega \tfrac12 |\zeta(x)|^2 \dd x,
\]
while the (fully) viscously regularized system \eqref{RD-intro} additionally involves the quadratic potentials
\[
\calV_{2,\nu}(v) : = \int_\Omega \tfrac{\nu}2 \bbD \sig v : \sig v \dd x, \qquad \calH_{2,\nu} (\pi) = \int_\Omega \tfrac{\nu}2 |\pi|^2 \dd x .
\]
\subsubsection*{The gradient structure for system \eqref{RD-intro}} It was  proved in \cite[Lemma 3.3]{Crismale-Rossi} that for every $t\in [0,T]$ 
the functional $(u,z,p) = q \mapsto \calE_\mu(t,q)$ is Fr\'echet differentiable on its domain $[0,T]\times \mathbf{D} $, with $\mathbf{D}= \{ (u,z,p) \in \Qha\, : \ z>0 \text{ in }
 \overline{\Omega}\}$, and that  for all $q \in \Qha$ the function $t\mapsto \calE_\mu(t,q)$ belongs to  $H^1(0,T)$. Relying on this,  it was shown  that system \eqref{RD-intro} reformulates as the generalized gradient system 
 \begin{equation}
\label{1509172300} 
\partial_{q'} \psie\eps\nu(q(t), q'(t)) + \rmD_q\enen{\mu} (t,q(t)) \ni 0 \qquad \text{in } \Qha^* \quad \foraa\, t \in (0,T),
\end{equation}
involving the overall dissipation potential
\[
  \psin\nu(q, q') 
  : = \disv 2\nu(u') + \calR(z') + \calR_2(z')+ \calH(z,p') + \calH_{2,\nu}(p')  
\]
{and its rescaled version }
\[
   \psie\eps\nu( q, q') : = \frac1{\eps} \psin\nu(q,\eps q')  = \int_{\Omega} \tfrac{\eps\nu}2 \bbD \sig{u'}{:} \sig{u'} \dd x + \calR(z')
   + \int_{\Omega} \tfrac\eps2 |z'|^2 \dd x + \calH(z,p') +  \int_{\Omega} \tfrac{\eps\nu}2 |p'|^2 \dd x .
\]

\section{The partial versus the  full vanishing-viscosity approach}
\label{s:3}
In this section we aim to gain further insight into the notion of  $\BVZZZ$  solution to system \eqref{RIS-intro} arising from the \emph{full} vanishing-viscosity approach of \cite{Crismale-Rossi} and compare it with the  
 $\BVZ$  concept from \cite{Crismale-Lazzaroni}. In order to properly introduce both notions, it can be useful to recall the reparameterized energy-dissipation balance where one passes to the limit  to obtain Balanced Viscosity solutions.
\par
 At this heuristic stage, we will treat the partial vanishing-viscosity approximation of \cite{Crismale-Lazzaroni} and the full approximation of \cite{Crismale-Rossi} in a unified way, although at the level of  the viscous approximation in  \cite{Crismale-Lazzaroni} there is the significant difference that the plastic strain evolves rate-independently. Still, we will disregard this and  instead focus on  the 
similarities between the \emph{rate-dependent} systems \eqref{RD-intro}
 (wherefrom the   $\BVZZZ$   solutions of  \cite{Crismale-Rossi}  stem), 
 and \eqref{viscous-CL-intro}  (wherefrom the  $\BVZ$  solutions of  \cite{Crismale-Lazzaroni}). In fact, 
 \eqref{viscous-CL-intro}
  can be formally obtained from   \cite{Crismale-Rossi} by setting $\nu =\mu=0$. That is why, in what follows
  to fix the main ideas 
    we will illustrate the limit passage in the energy-dissipation balance associated with system \eqref{RD-intro}. 
   \par
 Throughout this section and the remainder of the paper, we will suppose that the assumptions of Sec.\ \ref{s:2} are in force without explicitly invoking them. 
\subsection*{The energy-dissipation balance for the viscous system}
As observed in \cite[Proposition 3.4]{Crismale-Rossi}, 
 a curve $q = (u,z,p)\in H^1(0,T;\Qha)$
 is a solution to  \eqref{RD-intro}, namely to the generalized gradient system 
 \eqref{1509172300}, 
  if and only if  it satifies the energy-dissipation balance
\begin{equation}
\label{endissbal}
\int_0^t \left( \psie\eps\nu(q(r),q'(r)) + \psie\eps\nu^*(q(r),{-} \rmD_q\enen{\mu} (r,q(r))) \right) \dd r +   \enen{\mu} (t,q(t)) = \enen{\mu} (0,q(0))+\int_0^t \partial_t  \enen{\mu} (r,q(r)) \dd r 
\end{equation}
for every $t\in [0,T]$. 
Let us now consider a family $(q_{\eps,\nu}^\mu)$ of solutions to \eqref{1509172300}. In \cite{Crismale-Rossi} suitable
 a priori estimates, uniform w.r.t.\ the parameters $\eps, \nu, \mu>0$, for the length (measured in an appropriate norm) of the curves $q_{\eps,\nu}^\mu$   were derived.  Based on such estimates it is possible to reparameterize such curves, obtaining \emph{parameterized curves} defined on an `artificial' time interval $[0,S]$ 
 \[
 (\sft_{\eps,\nu}^\mu,\sfq_{\eps,\nu}^\mu): [0,S]\to [0,T] \times \Qha \quad \text{with } \sft_{\eps,\nu}^\mu : = (\sfs_{\eps,\nu}^\mu)^{-1}, \quad \sfq_{\eps,\nu}^\mu: = q_{\eps,\nu}^\mu \circ \sft_{\eps,\nu}^\mu
 \]
and $\sfs_{\eps,\nu}^\mu$  (suitable) arclength functions associated with the curves $(q_{\eps,\nu}^\mu)$. 
Now,  in terms of the parameterized curves  $(\sft_{\eps,\nu}^\mu,\sfq_{\eps,\nu}^\mu)$,   \eqref{endissbal} translates into the \emph{reparameterized} energy-dissipation balance 
\begin{equation}
\label{param-endissbal}
 \begin{aligned}
  & \calE (\sft_{\eps,\nu}^\mu(s),\sfq_{\eps,\nu}^\mu(s)) + \int_{0}^{s} 
   \Me{\sft_{\eps,\nu}^\mu(r)}{\sfq_{\eps,\nu}^\mu(r)}{(\sft_{\eps,\nu}^\mu)'(r)}{(\sfq_{\eps,\nu}^\mu)'(r)} \dd r
   \\
    & =  
      \calE (\sft_{\eps,\nu}^\mu(0),\sfq_{\eps,\nu}^\mu(0)) 
    +\int_{0}^{s} \partial_t \calE (\sft_{\eps,\nu}^\mu(r), \sfq_{\eps,\nu}^\mu(r)) \,(\sft_{\eps,\nu}^\mu)'(r) \dd r
    \end{aligned}
\end{equation}
featuring  the functional $\Mename : [0,T]\times \Qha \times (0,+\infty) \times \Qha \to [0,+\infty]$
\begin{align}
\label{Me-def}
&
\Me tq{t'}{q'} 
: = \calR(z') + \calH(z,p') + \Mredname (t,q,t',q')
\intertext{ with the reduced functional } \Mredname \text{ defined by }
&
\nonumber
\Mredname (t,q,t',q')
:=\  \frac{\eps}{2t'} \DVito_\nu(q')^2 + \frac{t'}{2\eps} (\DVitos(t,q))^2,
\end{align}
and 
\begin{equation}
\label{new-form-D-Dvito}
\begin{aligned}
  & 
   \DVito_\nu({q}'): = 
    \sqrt{ \nu \| {u}'(t)\|^2_{ H^1, \bbD }   {+} 
\|{z}'(t)\|_{L^2}^2
{+} \nu\|{p}'(t)\|_{L^2}^2} ,
\\
&
 \DVitos({t},{q}): =   \sqrt{
\frac1{\nu}\, \|{-}\mathrm{D}_u \calE_\mu ({t},{q} )\|^2_{( H^1, \bbD )^*} 
+ \tilded_{L^2} ({-}\mathrm{D}_z \calE_\mu  ({t},{q}) ,\partial\calR(0))^2
+ \frac1\nu \, \dLtwo ({-}\mathrm{D}_p \calE_\mu  ({t},{q}) ,\partial_\pi \calH( {z} ,0))^2
}\,.
\end{aligned}
\end{equation}
In \eqref{new-form-D-Dvito}, $\| \cdot \|_{ H^1, \bbD}  $ and $\| \cdot \|_{( H^1, \bbD )^*} $ are the norms  introduced   in \eqref{norma-equivalente} and \eqref{norma-duale}, while the distance functional $\tilded_{L^2(\Omega)}( \cdot, \partial \calR(0)): \Hs(\Omega)^* \to [0,+\infty]$ is defined by 
\begin{equation}
\label{tilde-d}
\tilded_{L^2(\Omega)}( { \chi}, \partial \calR(0))^2
  :=  \min_{\gamma \in \partial \calR(0)}   \mathfrak{f}_2 ( { \chi}{-}\gamma)
\quad \text{with }  \mathfrak{f}_2(\beta): = \begin{cases}
\|\beta\|_{L^2(\Omega)}^2 & \text{if } \beta \in L^2(\Omega) ,
\\
+\infty & \text{if } \beta \in \Hs(\Omega)^* \setminus L^2(\Omega) .
\end{cases}
\end{equation}
 Clearly, the 
functional $\Mename$ from \eqref{Me-def} encompasses, in the energy-dissipation balance \eqref{param-endissbal}, 
 the competition between viscous dissipation
 and tendency to relax towards equilibrium \& rate-independent behavior. In fact, viscous dissipation is 
 encoded in the term $   \DVito_\nu({q}')$, which is  modulated by the viscosity parameter $\eps$. In turn, the relaxation to rate-independent behavior is encompassed in the term $\DVitos({t},{q})$, modulated by $\tfrac1\eps$.
   
\par 
Now,  the concepts of  $\BVZ$ and $\BVZZZ$ solutions to system  \eqref{RIS-intro} 
are defined in terms of  parameterized energy-dissipation balances akin to \eqref{param-endissbal}. These energy identities involve a (positive) \emph{vanishing-viscosity potential} $\calM $ that is defined on $ [0,T]\times \Qpp \times [0,+\infty) \times \Qpp $
 (recall that $\Qpp$ is the state space for the perfectly plastic damage system, cf.\ \eqref{defQpp}). 
 At least  formally,
$\calM$
  arises as $\Gamma$-limit 
\begin{compactitem}
\item[-]
of the functionals $(\calM_{\eps}^{0,0})_\eps$  as $\eps\down0$,  in the case of the vanishing-viscosity analysis in \cite{Crismale-Lazzaroni} (recall that system \eqref{viscous-CL-intro} is formally a particular case of \eqref{RD-intro}, with $\nu=\mu=0$);
\item[-]
of $(\calM_{\eps}^{\mu,\nu})_{\eps,\nu,\mu}$  as $\eps, \, \nu,\, \mu \down0$,  in the case of the joint  vanishing-viscosity  and hardening analysis in~\cite{Crismale-Rossi}. 
\end{compactitem}
However, in order to \emph{rigorously} define the
vanishing-viscosity contact potentials
 $\calM$ relevant for the two concepts of $\BV$ solutions we will  need to provide some technical preliminaries
  in the following section. 
 \subsection{Preliminary definitions}
 \label{ss:3.prelim}
The
 vanishing-viscosity potentials at the core of the definitions of  $\BVZ$   solutions in  \cite{Crismale-Lazzaroni} and     $\BVZZZ$   solutions in 
 \cite{Crismale-Rossi} have a structure akin to that of the functional $\Mename$ from \eqref{Me-def},  
 but they are tailored to the driving energy $\calE_0$ \eqref{ene-plastic} for the perfectly plastic system.
  Recalling  the definition of  $\calM_{\eps}^{\mu,\nu}$,  it is thus clear that, in order to 
 define such vanishing-viscosity potentials,
 one  needs  suitable surrogates of 
  the   $(\bbD, H^1)^*$-norm   of  $\rmD_u \calE_0(t,q)$, and of the $L^2$-distance  of $\rmD_p \calE_0(t,q)$   from the stable set $\partial_\pi \mathcal{H}(z,0)$, cf.\ \eqref{new-form-D-Dvito}.  Indeed, such quantities 
 are no longer well defined for  the functional $\calE_0$  at every  $(t,q) \in [0,T]\times \Qpp$;  instead, note that 
  the 
 $L^2$-distance, in the sense   of  \eqref{tilde-d}, of    $\rmD_z \calE_0(t,q)$ 
  from the stable set $\partial \calR(0)$ still makes sense. 
Following \cite[Sec.\ 7]{Crismale-Rossi},  we will set at every $(t,q) \in  \mathrm{dom}(\calE_0) =[0,T]\times \mathbf{D}$ (cf.\ \eqref{domE0}) 
\begin{subequations}
\label{surrogates}
 \begin{align}
 \label{surrogate-1}
 &
 \calS_u \calE_0(t, q) 
:= \sup_{\substack{\eta_u \in H_\Dir^1(\Omega;\R^n) \\ \|\eta_u\|_{( H^1, \bbD)}\le 1} } \langle -\mathrm{Div}\,\sigma(t) -F(t), \eta_u \rangle_{ H_\Dir^1(\Omega;\R^n) }
, 
\\
&
 \label{surrogate-2}
\mathcal{W}_p \calE_0(t,q):=\sup_{ \substack{\eta_p \in L^2(\Omega;\MD) \\  \|\eta_p\|_{L^2}\le1} } \left(  \langle \sigma_\dev(t), \eta_p \rangle_{L^2(\Omega;\MD)} -\calH(z,\eta_p)\right) .
\end{align}
\end{subequations}
with  $\sigma(t)= \mathbb{C}(z) e(t) $ and $e(t) =  \rmE(u+w(t))-p$. 
 Observe that  the  expressions  in \eqref{surrogate-1} and \eqref{surrogate-2} are well defined since
$e(t)$ and, a fortiori, $\sigma(t)$ are elements in $L^2(\Omega;\Mnn)$ whenever $(u,z,p) \in \Qpp$. 
Formulae  \eqref{surrogate-1} and  \eqref{surrogate-2}  have been inspired by 
the obvious fact that 
\begin{equation*}
\| \zeta \|_{ (H^1, \bbD)^*} = \sup_{\substack{\eta_u \in H_\Dir^1(\Omega;\R^n) \\ \|\eta_u\|_{( H^1, \bbD)}\le 1} } \langle \zeta, \eta_u \rangle_{ H_\Dir^1(\Omega;\R^n) } \qquad \text{for every } \zeta \in H_\Dir^1(\Omega;\R^n)^*,
\end{equation*}
and by the formula
\begin{equation*}
d_{L^2}(\varsigma, \partial_\pi \mathcal{H}(z,0)) = \sup_{ \substack{\eta \in L^2(\Omega;\MD) \\  \|\eta\|_{L^2}\le1} } \left(  \langle \varsigma, \eta \rangle_{L^2(\Omega;\MD)} -\calH(z,\eta)\right) \qquad \text{for every } \varsigma \in  L^2(\Omega;\MD),
\end{equation*}
which was proved in \cite[Lemma 3.6]{Crismale-Rossi}.
In particular, we note that for every $q\in \mathrm{D}$
\begin{subequations}
\label{true-interpretation}
\begin{align}
 & \mathcal{W}_p \calE_0(t,q) = d_{L^2}(\sigma_\dev(t), \partial_\pi \mathcal{H}(z,0)),
 \intertext{while}
 &
 \calS_u \calE_0(t, q)  =\|  -\mathrm{Div}\,\sigma(t) -F(t) \|_{ (H^1, \bbD)^*}  \qquad \text{if } u \in H_\Dir^1(\Omega;\R^n).
\end{align}
\end{subequations}
For later use, we also record here the following result.
\begin{lemma}\label{le:0103222042}
For every $(t,q)\in [0,T]\times \mathbf{D}$ there holds
\begin{equation}
\label{equivalence}
\begin{gathered}
  \calS_u \calE_0(t, q) =\mathcal{W}_p \calE_0(t,q)=0 
\quad   \Longleftrightarrow 
 \\
 \sigma(t) \in \calK_{z}(\Omega)\,,\quad - \diver\, \sigma(t)=f(t) \text{ a.e.\ in }\Omega\,,\quad [\sigma(t) \rmn] = g(t) \ \hn\text{-a.e.\ on }\Gneu,
\end{gathered}
\end{equation}
with $\sigma(t) = \bbC(z) e(t)$ and $e(t) = \sig{e+w(t)} -p$.
\end{lemma}
\begin{proof}   The  implication $\Rightarrow $ was proved in \cite[Lemma 7.4]{Crismale-Rossi}. A close perusal of the proof, also mimicking the convexity  arguments
from \cite[Prop.\ 3.5]{DMDSMo06QEPL},  
 also yields the converse implication. 
\end{proof} 
\par
We then set 
\begin{equation}
\label{right-Dstar}
\DVito^*(t,q): = \sqrt{( \calS_u \calE_0(t, q) )^2 +(\mathcal{W}_p \calE_0(t,q))^2} .
\end{equation}
We are now in a position to define the vanishing-viscosity contact potentials involved in the definitions of $\BV$ solutions in \cite{Crismale-Lazzaroni} and \cite{Crismale-Rossi}. We will use the notation $ I_{\{0\}}$ for the indicator function of the singleton $\{0\}$, namely
\[
I_{\{0\}}(\xi) = \begin{cases}
0 & \text{if } \xi =0,
\\
+\infty & \text{otherwise}
\end{cases} \qquad \text{for all } \xi \in \R.
\]
\subsection{$\BVZ$  solutions via the partial vanishing-viscosity approach}
\label{ss:3.1}
The vanishing-viscosity contact potential for the $\BVZ$  solutions from \cite{Crismale-Lazzaroni} is the functional 
\begin{subequations}
\label{M0-CL}
\begin{equation}
\label{M0-CL-1}
\calM_0: [0,T] \times \Qpp \times [0,+\infty) \times \Qpp \to [0,+\infty], \qquad 
\calM_0(t,q,t',q') =  \calR(z') + \Hpp(z,p') + \calM_{0, \mathrm{red}}(t,q,t',q') ,
\end{equation}
where for $q=(u,z,p)$ and $q'=(u',z',p')$ we have 
\begin{equation}
\label{M0-CL-2}
 \calM_{0, \mathrm{red}}(t,q,t',q')  = I_{\{ 0\}} (\DVito^*(t,q)) +  \widetilde{\calM}_{0, \mathrm{red}}(t,q,t',q')  
\end{equation}
and 
\begin{equation}
\label{2906191736}
\begin{aligned}
& 
\text{if } t'>0,  \quad 
 \widetilde{\calM}_{0, \mathrm{red}}(t,q,t',q'):= I_{\{0\}}(\tilded_{L^2} ({-}\rmD_z \calE_0 (t,q),\partial\calR(0))),
\\
&
\text{if } t'=0,  \quad 
 \widetilde{\calM}_{0, \mathrm{red}}(t,q,t',q') :=
 \begin{cases}
\|z'\|_{L^2} \, \tilded_{L^2}  ({-}\mathrm{D}_z \calE_0 (t,q),\partial\calR(0))  & \text{if } \tilded_{L^2}  ({-}\mathrm{D}_z \calE_0 (t,q),\partial\calR(0))<+\infty,
\\
+\infty & \text{otherwise}
\end{cases}\,.
\end{aligned}
\end{equation}
\end{subequations}
Observe that the definition of $\calM_0(t,q,0,q')$
again reflects,  in the jump regime,  the tendencies of the system to  evolve viscously  and to relax towards equilibrium and rate-independent evolution. However, here viscous dissipation only affects the damage variable. Likewise,  rate-independent behavior is solely encompassed by  the relation $\tilded_{L^2} ({-}\mathrm{D}_z \calE_0 (t,q),\partial\calR(0)) =0$. 
\par
Let us now detail the properties of the parameterized curves providing $\BVZ$ solutions. 
\begin{definition}
\label{def:CL16}
We call a parameterized  curve   $(\sft,\sfq)=(\sft, \sfu, \sfz, \sfp)\colon [0,S] \to [0,T] \times \Qpp$ \emph{admissible} if it satisfies 
\begin{subequations}\label{CL-admiss-curve}
\begin{align}
&\text{$\sft \colon [0,S] \to [0,T]$ is nondecreasing,}
\label{CL-adm-def-1}
\\
&(\sft, \sfu, \sfz, \sfp) \in \AC\big([0,S]; [0,T] {\times} \BD(\Omega) {\times} \Hs(\Omega)  {\times} \MbD \big),\label{CL-adm-def-2}\\
&  \sfe=\mathrm{E}(\sfu + w(\sft))-\sfp \in \AC([0,S]; \Lnn),  
    \label{CL-adm-def-3}\\
& 
 \label{CL-adm-def-4}
\sft \text{ is constant in every connected component of }A^\circ = \{ s\in (0,S) \colon \tilded_{L^2} ({-}\mathrm{D}_z \calE_0 (\sft(s),\sfq(s)),\partial\calR(0))  >0\}.
\end{align}
\end{subequations}
We will denote by $\adm 0S 0T{\Qpp}$ the class of admissible parameterized curves from $ [0,S] $ to $[0,T] \times \Qpp$.
\end{definition}
We are now in a position to give the definition of Balanced Viscosity solution in the sense of \cite{Crismale-Lazzaroni}.
\begin{definition}
\label{def:BV-sol-CL}
We call an admissible parameterized   curve  $(\sft,\sfq)=(\sft, \sfu, \sfz, \sfp)\in \adm 0S 0T{\Qpp}$  a 
Balanced Viscosity solution   for the perfectly plastic damage system
\eqref{RIS-intro} in the sense of \cite{Crismale-Lazzaroni}
(a \emph{$\BVZZZ $ solution}, for short), 
if it satisfies 
 the energy-dissipation balance
\begin{equation}
\label{EnDiss-CL}
\calE_0(\sft(s_2),\sfq(s_2)) + \int_{s_1}^{s_2} \calM_0(\sft(r),\sfq(r),\sft'(r),\sfq'(r))
 \dd r =  \calE_0(\sft(s_1),\sfq(s_1)) + \int_{s_1}^{s_2} \partial_t \calE_0(\sft(r), \sfq(r)) \dd r
   \end{equation}
   for every $0 \leq s_1 \leq s_2 \leq S$. 
\end{definition}
%
%
The existence  of $\BVZ$ solutions was proved in  \cite[Thm.\ 5.4]{Crismale-Lazzaroni}  
under the condition that the initial data $q_0=(u_0, z_0, p_0)$ for the perfectly plastic damage system 
\eqref{RIS-intro} fulfill \eqref{init-data} and the additional condition that 
\begin{equation}
\label{EL-initial-CL}
\begin{aligned}
\rmD_q  \enen{0} (0,q_0)  =
 \left(-\mathrm{Div}\, \sigma_0  -F(0), \As(z_0)+W'(z_0)+\tfrac12 \mathbb{C}'(z_0) e_0: e_0, 
\newmu p_0 - (\sigma_0)_\dev \right) \in L^2(\Omega;\R^n{\times} \R {\times} \MD) .
\end{aligned}
\end{equation}

\subsection{A differential characterization for  $\BVZ$ solutions}
We  now 
aim to provide a \emph{differential characterization}  for the  notion of $\BV$ solution 
 from Definition \ref{def:CL16},  in terms of a suitable system of 
 subdifferential inclusions for the displacement variable (which in fact shall satisfy the elastic equilibrium equation), the damage variable, and the plastic strain. In order to properly formulate the flow rule governing the latter, we need   the following result;   the proof of  one implication can be  found \cite{Crismale-Lazzaroni}, in turn based on arguments from \cite{DMDSMo06QEPL}. 
 \begin{lemma}
 \label{l:plastic-flow-rule}
 Let an admissible parameterized curve  $(\sft,\sfq)=(\sft, \sfu, \sfe, \sfp) \in \adm 0S0T{\Qpp}$ satisfy 
 \begin{equation}
 \label{mHDP}
   \Hpp(\sfz(s),\sfp'(s)) =   \langle \serifsigma_\dev(s) \, |\, \sfp'(s) \rangle  \qquad \foraa\, s \in (0,S)
 \end{equation}
 with  $ \serifsigma(s) = \bbC(\sfz(s)) \sfe(s)$, $\sfe(s) = e (\sft(s))$, and $\langle \cdot  \, |\,  \cdot \rangle $ the stress-strain duality from \eqref{duality-product}. 
\par
 Then, $(\sft,\sfq)$ satisfies   \emph{Hill's maximum work principle}
 \begin{equation}
 \label{HmWP}
 H\left(\sfz(s), \frac{\dd {\sfp}'(s)}{\dd |\sfp'(s)|}\right)|\sfp'(s)| = [\serifsigma_\dev(s):\sfp'(s)] \qquad \foraa\, s \in (0,S),
 \end{equation}
  where the above equality holds in the sense of  measures on  $ \Omega \cup \Gdir $, 
 with $ [\serifsigma_\dev(s):\sfp'(s)] $ the distribution defined in \eqref{sD}.  Furthermore,
 defining 
 $\mu(s): = \mathcal{L}^n + | \sfp'(s)|$ for every $s\in [0,S]$, there exists $\widehat{\serifsigma}_\dev(s) \in L_{\mu(s)}^\infty (\Omega {\cup} \Gdir;\MD)$ such that for almost all $s\in (0,S)$ the following properties hold:
 \begin{subequations}
 \label{eq:lambda-CL-3-precise}
 \begin{align}
 &
 \widehat{\serifsigma}_\dev(s)  = \serifsigma_\dev(s)  && \mathcal{L}^n\text{-a.e.\ on } \Omega,
 \\
 &
  [\serifsigma_\dev(s):\sfp'(s)]  = \left( \widehat{\serifsigma}_\dev(s){:} \frac{\dd \sfp'(s)}{\dd |\sfp'(s)|} \right)   |\sfp'(s)|  && \text{on }  \Omega \cup \Gdir,
  \\
  &
  \partial_\pi H\left(\sfz(s), \frac{\dd \sfp'(s)}{\dd |\sfp'(s)|}  \right) \ni \widehat{\serifsigma}_\dev(s) && \text{for }   |\sfp'(s)|\text{-a.e. } x \in  \Omega \cup \Gdir.
 \end{align} 
 \end{subequations}
Conversely, \eqref{eq:lambda-CL-3-precise}  imply \eqref{HmWP} which, in turn, gives \eqref{mHDP}. 
  \end{lemma}
  \begin{proof}
 We refer to \cite[Prop.\ 6.5]{Crismale-Lazzaroni} for the proof of the fact that \eqref{mHDP} implies \eqref{HmWP} and \eqref{eq:lambda-CL-3-precise}. In turn, recalling \eqref{charact-calR}, from (\ref{eq:lambda-CL-3-precise}c) we infer that 
  \[
   H\left(\sfz(s), \frac{\dd \sfp'(s)}{\dd |\sfp'(s)|}  \right)   = \left( \widehat{\serifsigma}_\dev(s){:} \frac{\dd \sfp'(s)}{\dd |\sfp'(s)|} \right)  \qquad  |\sfp'(s)|\text{-a.e.\ in} \,  \Omega \cup \Gdir.
  \]
  Combining this with  (\ref{eq:lambda-CL-3-precise}a)  and  (\ref{eq:lambda-CL-3-precise}b) we  conclude  \eqref{HmWP}, which yields  \eqref{mHDP}  in view of the definition \eqref{plast-diss-pot-visc} of $\Hpp$, and of the definition  \eqref{duality-product} of the stress-strain duality product. 
  \end{proof}
  
  \noindent
 It is in the sense of \eqref{eq:lambda-CL-3-precise}
 that we need to  understand the (formally written)  inclusion $\partial_\pi H(\sfz, \sfp'\ni\serifsigma_\dev$.
 \par
 We are now in a position to 
prove the following differential characterization
for the concept of $\BV$ solution from \cite{Crismale-Lazzaroni}. 
We mention in advance that,  for notational simplicity,  in \eqref{eq:lambda-CL-2} we have simply written $ \sfz'(s) $ in place of $J( \sfz'(s) )$, with $J: L^2(\Omega) \to \Hs(\Omega)^*$ the Riesz operator. 
\begin{proposition}
\label{diff-charact-BV-CL}
An admissible parameterized curve  $(\sft,\sfq)\in \adm 0S0T{\Qpp}$ is a   $\BVZ$ solution to system \eqref{RIS-intro}   
if and only if there exists a measurable function $\lambda_{\mathsf{z}}: [0,S]\to [0,1]$ such that 
\begin{align}
&
\label{switch-CL-1}
\sft'(s) \lambda_{\sfz}(s) =0  &&  \foraa\ s\in(0,S),
 \end{align}
and $(\sft,\sfq)$ satisfies \foraa\ $s \in (0,S)$
\begin{subequations}
\label{differential-chact-CL}
\begin{align}
& 
\label{eq:lambda-CL-1}
 \serifsigma(s) \in \calK_{\sfz(s)}(\Omega),\quad - \diver\, \serifsigma(s)=f(\sft(s)) \text{ a.e.\ in }\Omega,\quad [\serifsigma(\sft(s)) \rmn] = g(\sft(s)) \ \hn\text{-a.e.\ on }\Gneu,
 \\
&
\label{eq:lambda-CL-2}
(1{-}\lambda_{\sfz}(s))\, \partial \mathcal{R}(\sfz'(s))+ \lambda_{\sfz}(s)\, \sfz'(s) + (1{-}\lambda_{\sfz}(s)) \,\left[ \As \sfz(s){+}W'(\sfz(s)){+}\tfrac12 \bbC'(\sfz(s)) \sfe(s): \sfe(s) \right]   \ni 0 \quad
 \text{in } 
\Hs(\Omega)^* ,
\\
&
\label{eq:lambda-CL-3}
\partial_\pi \Hpp(\sfz(s), \sfp'(s)) \ni \serifsigma_\dev(s) \quad \text{in the sense of \eqref{eq:lambda-CL-3-precise}}.
\end{align}
\end{subequations}
 In fact, \eqref{eq:lambda-CL-1} holds at every $s\in (0,S]$.
\end{proposition}
\begin{remark}
\upshape
System \eqref{differential-chact-CL} illustrates in a clear way how viscous behavior, \emph{only w.r.t.\ the variable $z$},  may arise in the jump regime, namely when  the system still evolves while $\sft'=0$ (i.e.\ the external time, recorded by the function $\sft$, is frozen). In that case, the parameter $\lambda_{\sfz}$ may be non-zero, thus activating the viscous contribution to  the flow rule
\eqref{eq:lambda-CL-2}. 
\end{remark}

\par
Prior to carrying out the proof of  Proposition \ref{diff-charact-BV-CL},   we record the following key chain-rule inequality, whose proof may be immediately inferred from that of   \cite[Lemma 7.6]{Crismale-Rossi}
(cf.\ also Lemma \ref{l:ch-rule-pp-CR} ahead).
 \begin{lemma}
 \label{l:ch-rule-pp-CL}
     Along any admissible parameterized curve 
    \begin{align}
    &
    \nonumber
(\sft,\sfq) \in \adm 0S0T{\Qpp}
\text{ s.t. }    \mathcal{M}_0(\sft,\sfq,\sft',\sfq')<+\infty \ \aein\, (0,S), 
 \intertext{ we have that $s \mapsto   \calE_{0}(\sft(s),\sfq(s)) $ is absolutely continuous on $[0,S]$ 
and there holds } 
&    
    \label{desired-chain-rule-pp}
    \begin{aligned}
    &
-\frac{\dd}{\dd s} \calE_{0}(\sft(s),\sfq(s)) +\partial_t  \calE_{0}(\sft(s),\sfq(s))\,  \sft'(s)  
 \\
  & = - \langle \rmD_z \calE_0(\sft(s),\sfq(s)), \sfz'(s) \rangle_{\Hs}  + \langle \serifsigma_\dev(s) \, |\, \sfp'(s) \rangle + \langle  \Diver\,\serifsigma(s) {+}  F(\sft(s)),  \sfu'(s) \rangle_{\BD}
  \\
  &
 \leq \mathcal{M}_0(\sft(s),\sfq(s),\sft'(s),\sfq'(s)) \qquad  \foraa\, s \in (0,S).
\end{aligned}
    \end{align}
 \end{lemma}
  As an immediate consequence of Lemma \ref{l:ch-rule-pp-CL} we have the following characterization of 
 the $\BV_0$ solutions from Definition  \ref{def:CL16}, which complements the other characterizations provided in \cite[Prop.\ 5.3]{Crismale-Lazzaroni}.
 \begin{corollary}
 \label{cor:further-charact}
 An admissible parameterized curve  $(\sft,\sfq)\in \adm 0S0T{\Qpp}$ is a $\BVZ$ solution to system \eqref{RIS-intro}
 if and only if  the function
 $[0,S] \ni s \mapsto   \calE_{0}(\sft(s),\sfq(s)) $  is absolutely continuous and there holds
 \begin{equation}
\label{BV-as-equality}
\begin{aligned}
&    
-\frac{\dd}{\dd s} \calE_{0}(\sft(s),\sfq(s)) +\partial_t  \calE_{0}(\sft(s),\sfq(s))\,  \sft'(s)  
 \\
  & = - \langle \rmD_z \calE_0(\sft(s),\sfq(s)), \sfz'(s) \rangle_{\Hs}  + \langle \serifsigma_\dev(s) \, |\, \sfp'(s) \rangle + \langle  \Diver\,\serifsigma(s) {+}  F(\sft(s)),  \sfu'(s) \rangle_{\BD}
 = \mathcal{M}_0(\sft(s),\sfq(s),\sft'(s),\sfq'(s))
\end{aligned}
\end{equation}
 for almost all $s\in (0,S)$.
 \end{corollary}
 \par
 We are now in a position to carry out the 
\begin{proof}[Proof  of Proposition \ref{diff-charact-BV-CL}]
By Corollary \ref{cor:further-charact}, $\BV_0$ solutions in the sense of Def.\ \ref{def:CL16} can be characterized in terms of \eqref{BV-as-equality}, whence we deduce that 
$\mathcal{M}_0(\sft(s),\sfq(s),\sft'(s),\sfq'(s))<\infty$ for almost all $s\in (0,S)$. Therefore, 
 \begin{equation}
 \label{DVs=0}
 \DVito^*(\sft(s),\sfq(s)) =0 \qquad \foraa\, s \in (0,S).
 \end{equation}
 In turn, \eqref{DVs=0} is equivalent to the validity of 
 \eqref{eq:lambda-CL-1}.
  Observe that   \eqref{eq:lambda-CL-1} extends to every $s\in [0,S]$ by the continuity of the functions
$s \mapsto \serifsigma(s) $, $s\mapsto f(\sft(s)) $, $s\mapsto g(\sft(s))$ and by the continuity 
w.r.t.\ the Hausdorff distance of $s\mapsto  \calK_{\sfz(s)}(\Omega)$ thanks to \eqref{Ksets-3}. 
\par
   In view of \eqref{DVs=0} and recalling the definition 
   \eqref{M0-CL} of $\calM_0$, 
    \eqref{BV-as-equality} can be then rewritten as 
   \[
  - \langle \rmD_z \calE_0(\sft(s),\sfq(s)), \sfz'(s) \rangle_{\Hs} -   \calR(\sfz'(s)) - \widetilde{\calM}_{0, \mathrm{red}}(\sft(s),\sfq(s),\sft'(s),\sfq'(s)) =    \Hpp(\sfz(s),\sfp'(s)) -   \langle \serifsigma_\dev(s) \, |\, \sfp'(s) \rangle   
  \]
  for a.a.\ $s\in (0,S)$.
  Now, on the one hand the right-hand side of the above equality is positive thanks to \eqref{eq:carH}. On the other hand, let $\omega \in \partial \calR(0) $ satisfy 
  $\tilded_{L^2} ({-}\mathrm{D}_z \calE_0 (t,q),\partial\calR(0))  =   \|{-} \rmD_z \calE_0(\sft(s),\sfq(s)){-}\omega\|_{L^2} $. Then, we may estimate the left-hand side of the above equality via 
  \[
  \begin{aligned}
  &
    \langle {-} \rmD_z \calE_0(\sft(s),\sfq(s)), \sfz'(s) \rangle_{\Hs} -   \calR(\sfz'(s)) - \widetilde{\calM}_{0, \mathrm{red}}(\sft(s),\sfq(s),\sft'(s),\sfq'(s)) 
    \\
    & 
    = \dddn{ \langle {-} \rmD_z \calE_0(\sft(s),\sfq(s)){-}\omega, \sfz'(s) \rangle_{\Hs}  - \widetilde{\calM}_{0, \mathrm{red}}(\sft(s),\sfq(s),\sft'(s),\sfq'(s))}{$T_1$}   + \dddn{\langle \omega, \sfz'(s) \rangle_{\Hs} -   \calR(\sfz'(s))}{$T_2$}  \leq 0,
    \end{aligned}
  \]
  where we have used that $T_1 \leq 0$ by the very definition of $\widetilde{\calM}_{0, \mathrm{red}}$ (cf.\ \eqref{2906191736}), while $T_2\leq 0$ thanks to \eqref{charact-calR}. 
Therefore, \eqref{BV-as-equality}
 is ultimately \emph{equivalent} to 
  \begin{subequations}
  \label{separate-identities} 
  \begin{align}
   &
 \label{separate-identities-1} 
     \langle{-} \rmD_z \calE_0(\sft(s),\sfq(s)) {-} \omega , \sfz'(s) \rangle_{\Hs}  = 
    \widetilde{\calM}_{0, \mathrm{red}}(\sft(s),\sfq(s),\sft'(s),\sfq'(s))  && \foraa\, s \in (0,S),
\\
   &
 \label{separate-identities-1+2}     
    \langle \omega, \sfz'(s)  \rangle_{\Hs} =    \calR(\sfz'(s))   && \foraa\, s \in (0,S),
    \\
    &
     \label{separate-identities-2} 
     \langle \serifsigma_\dev(s) \, |\, \sfp'(s) \rangle  =  \Hpp(\sfz(s),\sfp'(s)) && \foraa\, s \in (0,S)
  \end{align}
    \end{subequations}
    for every $\omega \in \partial \calR(0)$ such that  $ \|{-} \rmD_z \calE_0(\sft(s),\sfq(s)){-}\omega\|_{L^2} = \tilded_{L^2(\Omega)} ({-}\mathrm{D}_z \calE_0 (t,q),\partial\calR(0))  $. 
    
Now, the same argument as in, e.g., \cite[Prop.\ 5.1]{KRZ13} 
allows us to infer that
  \eqref{separate-identities-1} and \eqref{separate-identities-1+2}       are  \emph{equivalent} to 
  \eqref{eq:lambda-CL-2},  together with \eqref{switch-CL-1}.  In turn,  
 by  Lemma \ref{l:plastic-flow-rule}
 it follows that
  \eqref{separate-identities-2}
 is \emph{equivalent} to   \eqref{eq:lambda-CL-3-precise}.  
 All in all, we have shown the equivalence between \eqref{BV-as-equality} and  \eqref{switch-CL-1}\&\eqref{differential-chact-CL}.  This concludes the proof. 
 \end{proof}
 
\subsection{$\BV$ solutions via the full vanishing-viscosity approach}
\label{ss:3.2} 
The vanishing-viscosity contact potential for the $\BV$ solutions from \cite{Crismale-Rossi} is the functional 
$\oMliname: [0,T]  \times \Qpp \times [0,+\infty) \times \Qpp \to [0,+\infty]$ 
defined via 
\begin{subequations}\label{M0-CR}
\begin{equation}\label{2906191728}
\oMliname(t,q,t',q'):=   \calR(z') + \Hpp(z,p') + \oMliredname(t,q,t',q') 
\end{equation}
where for $q=(u,z,p)$ and $q'=(u',z',p')$ we have 
\begin{equation}
\label{2906191736'}
\begin{aligned}
& 
\text{if } t'>0,  \quad 
\oMliredname(t,q,t',q'):= 
\begin{cases}
 0 &\text{if } 
\begin{cases}
\calS_u \calE_0 (t,q)=0 ,  \\
\tilded_{L^2} ({-}\rmD_z \calE_0 (t,q),\partial\calR(0))=0
, \ \text{and} \\
\calW_p\calE_0(t,q)=0 ,
\end{cases}
 \\
 +\infty & \text{otherwise,}
\end{cases}
\end{aligned}
\end{equation}
\begin{equation}
\label{oMlizero=0}
\begin{aligned}
\text{if } t'=0,  \quad 
 \oMliredname(t,q,t',q'): =
\begin{cases}
   \DVito(u',p') \, \DVito^*(t,q)  & \text{if } z'=0, 
\\
\|z'\|_{L^2}\tilded_{L^2} ({-}\mathrm{D}_z \calE_0 (t,q),\partial\calR(0)) & \text{if } \DVito^*(t,q)=0 
\\
 & \text{ and } \tilded_{L^2}  ({-}\mathrm{D}_z \calE_0 (t,q),\partial\calR(0))<+\infty, 
\\
+\infty &  \text{otherwise}, 
\end{cases}
\end{aligned}
\end{equation}
with 
\begin{equation} \label{def:DVito}
  \DVito(u', p'): =   \sqrt{ \|u'\|_{H^1,\bbD}^2 {+} \|p'\|_{L^2(\Omega;\MD)}^2} \quad \text{ and } \quad 
  \DVito^* \text{ from \eqref{right-Dstar}.}
 \end{equation}
\end{subequations}
As previously remarked, for every $(t,q) \in [0,T]\times \Qpp$  we have that $ \calS_u \calE_0 (t,q)<+\infty$
and $\calW_p\calE_0(t,q)<+\infty$, hence $\DVito^*(t,q)<+\infty$ and 
the product $\DVito(u',p') \, \DVito^*(t,q)$ is well defined as soon as $u'\in H^1(\Omega;\R^n)$ and $p'\in L^2(\Omega;\MD)$; otherwise, we mean $\DVito(u',p') \, \DVito^*(t,q)=+\infty$.
\par
The (reduced) vanishing-viscosity contact potentials 
 $\calM_{0, \mathrm{red}}$  
 from \eqref{M0-CL-2}
 and  $\calM^{0,0}_{0,\mathrm{red}}$ differ
 from each other, both in their definition 
 for $t'>0$ and for $t'=0$. For $t'>0$,  $\calM^{0,0}_{0,\mathrm{red}}$ has to additionally enforce elastic equilibrium
 (i.e.\  $\calS_u \calE_0 (t,q)=0$) and the  \emph{stability} constraint that $\sigma_\dev \in \partial_\pi \mathcal{H}(z,0)$ (i.e.\ 
 $\calW_p\calE_0(t,q)=0 $)
 since, for the fully rate-dependent viscous systems, these costraints are no longer fulfilled.
Accordingly, viscous behavior in $u$ and $p$ may intervene in the jump regime of the rate-independent
limit system. This is encoded in the new term $ \DVito(u',p') \, \DVito^*(t,q) $ featuring in 
\eqref{oMlizero=0}, which appears in the energy-dissipation balance at jumps (i.e.\ for $t'=0$),  when $z'=0$. 
\par
 The following definition   specifies the properties of the parameterized curves that are $\BVZZZ$ solutions
and is to be compared with Definition \ref{def:CL16} of admissible parameterized curves in the sense of 
\cite{Crismale-Lazzaroni}.  
\begin{definition}
We call a parameterized  curve   $(\sft,\sfq)=(\sft, \sfu, \sfz, \sfp)\colon [0,S] \to [0,T] \times \Qpp$ \emph{admissible in an enhanced sense}
(`enhanced admissible' for short) if it satisfies \eqref{CL-adm-def-1}, \eqref{CL-adm-def-2},  \eqref{CL-adm-def-3} and, in addition,
\begin{subequations}\label{CR-admiss-curve}
\begin{align}
& (\sfu, \sfp) \in \AC_{\mathrm{loc}}(B^\circ; H^1(\Omega;\R^n){\times} L^2(\Omega; \MD)), \text{ where }B^\circ:=\{ s\in (0,S) \colon  \mathcal{D}^*(\sft(s), \sfq(s))>0\},  \label{2906191815}\\
& 
 \label{2906191815+1}
\sft \text{ is constant in every connected component of }B^\circ.
\end{align}
\end{subequations}
We will denote by $\eadm 0S 0T{\Qpp}$ the class of  enhanced admissible parameterized curves from $ [0,S] $ to $[0,T] \times \Qpp$.
\end{definition}
 Hence, enhanced admissible curves enjoy better spatial regularity, with $\sfu(\cdot) \in H^1(\Omega;\R^n)$
and $\sfp (\cdot) \in L^2(\Omega; \MD)$, in the set in which  either $\calS_u \calE_0 (\sft(\cdot),\sfq(\cdot))>0$
or $\calW_p\calE_0(\sft(\cdot),\sfq(\cdot))>0 $. With that definition at hand, 
we are now in a position to give the definition of $\BV$ solution in the sense of \cite{Crismale-Rossi}.
\begin{definition}
\label{def:BV-sol-CR}
We call an enhanced admissible parameterized   curve  $(\sft,\sfq)=(\sft, \sfu, \sfz, \sfp)\in \eadm 0S 0T{\Qpp}$  a  Balanced Viscosity solution for the perfectly plastic damage system
\eqref{RIS-intro} in the sense of \cite{Crismale-Rossi}
(a \emph{$\BVZZZ$ solution},  for short),  
if it satisfies 
 the energy-dissipation balance
\begin{equation}
\label{EnDiss-CR}
\calE_0(\sft(s_2),\sfq(s_2)) + \int_{s_1}^{s_2} \calM^{0,0}_0(\sft(r),\sfq(r),\sft'(r),\sfq'(r))
 \dd r =  \calE_0(\sft(s_1),\sfq(s_1)) + \int_{s_1}^{s_2} \partial_t \calE_0(\sft(r), \sfq(r)) \dd r
   \end{equation}
   for every $0 \leq s_1 \leq s_2 \leq S$. 
\end{definition}
\noindent 
The existence of $\BVZZZ $ solutions to system \eqref{RIS-intro} was proved 
in \cite[Thm.\ 7.9]{Crismale-Rossi}
for initial data $q_0 = (u_0,z_0,p_0) $ complying with \eqref{init-data} and  \eqref{EL-initial-CL}.
\subsection{A differential characterization for $\BVZZZ$ solutions}
\label{ss:3.5}
In this section we provide a differential characterization  for $\BVZZZ$ solutions.
Preliminarily,  we need to make precise in which sense we are going to understand the subdifferential inclusions governing the evolution of the reparameterized displacement and of the plastic variables. 
Indeed, 
by  formally  writing
\[
 -\lambda  \mathrm{div}\,\bbD \sig{\sfu'} - (1{-} \lambda)   \mathrm{Div}\,\serifsigma  =(1{-} \lambda)    F(\sft) \quad \aein\, (0,S),
\]
with $\lambda: [0,S]\to [0,1]$ a measurable function  (below we will have $\lambda =\lambda_{\sfu,\sfp}$),
 we shall mean
\begin{subequations}
\label{sense-displacements}
\begin{align}
&
\label{sense-displacements-1}
-   \mathrm{Div}\,\serifsigma(s)  =   F(\sft(s)) \qquad  \text{in } \BD(\Omega)^* && \text{if } \lambda(s) =0,
\\
& 
\label{sense-displacements-2}
\begin{cases}
 \sfu'(s) \in H_\Dir^1(\Omega;\R^n),
\\
-\lambda(s)  \mathrm{div}\,\bbD \sig{\sfu'(s)} - (1{-} \lambda(s))   \mathrm{Div}\,\serifsigma(s) =(1 {-} \lambda(s))    F(\sft(s)) \qquad  \text{in } H^1(\Omega;\R^n)^*
\end{cases}
&& \text{if } \lambda(s) >0,
\end{align}
\end{subequations}
where in \eqref{sense-displacements} $ \mathrm{Div}\,\serifsigma(s)$ denotes the restriction of the functional from 
\eqref{2307191723} to $H^1(\Omega;\R^n)$. 
In particular, let us emphasize that, when $\lambda>0$ the  displacement variable enjoys additional spatial regularity, and the quasistatic momentum balance \eqref{sense-displacements-2} allows for test functions in $H^1(\Omega;\R^n)$. 
Likewise,  by 
writing
\[
(1{-} \lambda)  \partial_{\pi} \Hpp(\sfz, \sfp') +  \lambda \sfp' \ni 
(1{-} \lambda) \serifsigma_\dev
\]
with $\lambda: [0,S]\to [0,1]$ a measurable function, we shall mean
\begin{subequations}
\label{sense-pl}
\begin{align}
&
\label{sense-pl-1}
 \partial_{\pi} \Hpp(\sfz(s), \sfp'(s)) \ni 
 \serifsigma_\dev(s) \quad \text{in the sense of \eqref{eq:lambda-CL-3-precise} }  && \text{if } \lambda(s) =0,
\\
& 
\label{sense-pl-2}
\begin{cases} 
\sfp'(s) \in L^2(\Omega;\MD),
\\
(1{-} \lambda(s))  \partial_{\pi} \mathcal{H}(\sfz(s), \sfp'(s)) +  \lambda(s) \sfp'(s) \ni 
(1{-} \lambda(s)) \serifsigma_\dev(s) \quad \aein\, \Omega
\end{cases}
&& \text{if } \lambda(s) >0.
\end{align}
\end{subequations}
Namely, the plastic flow rule improves to a pointwise-in-space formulation in the set $\{ \lambda>0\}$, whereas in the set $\{\lambda=0\}$ it  only holds in the weak form \eqref{eq:lambda-CL-3-precise}. 
\par
We are now in a position to state our differential characterization of $\BVZZZ$ solutions.  Observe that 
the  definition of enhanced admissible curve is tailored to 
the subdifferential inclusions \eqref{diff-charact-BV-CR}. 

    \begin{proposition}
    \label{prop:diff-charact-BV-CR}
    An enhanced admissible parameterized curve 
$(\sft,\sfq)  \in \eadm 0S0T{\Qpp}$ is a  $\BVZZZ$ solution to system \eqref{RIS-intro} 
 if and only if there exist two  measurable 
 functions $\lambda_{\sfu,\sfp}\,  \lambda_{\sfz}: [0; S] \to [0; 1]$ such that 
\begin{subequations}
\label{switch-BV-CR}
\begin{align}
&
\label{switch-BV-CR-1}
\sft'(s) \lambda_{\sfu,\sfp}(s) =\sft'(s) \lambda_{\sfz}(s) =0  &&  \foraa\ s\in(0,S),
\\
& 
\label{switch-BV-CR-2}
 \lambda_{\sfu,\sfp}(s)(1{-} \lambda_{\sfz}(s))  = 0  &&  \foraa\ s\in(0,S), 
 \end{align}
\end{subequations}
and the curves $(\sft,\sfq)$ satisfy \foraa\ $s \in (0,S)$
the system of subdifferential inclusions 
\begin{subequations}
\label{diff-charact-BV-CR}
\begin{align}
\label{diff-charact-BV-CR-1}
& -\lambda_{\sfu,\sfp}(s)  \mathrm{div}\,\bbD \sig{\sfu'(s)} - (1{-} \lambda_{\sfu,\sfp}(s))   \mathrm{Div}\,\serifsigma(s)  =(1{-} \lambda_{\sfu,\sfp}(s))    F(\sft(s)),  
\\
&
\label{diff-charact-BV-CR-2}
(1{-}\lambda_{\sfz}(s))\, \partial \mathcal{R}(\sfz'(s))+ \lambda_{\sfz}(s)\, \sfz'(s) + (1{-}\lambda_{\sfz}(s)) \,\left(\As \sfz(s){+}W'(\sfz(s)){+}\tfrac12 \bbC'(\sfz(s)) \sfe(s): \sfe(s)\right)   \ni 0 \quad
 \text{in } 
\Hs(\Omega)^* ,
\\
&
\label{diff-charact-BV-CR-3}
(1{-} \lambda_{\sfu,\sfp}(s)) \, \partial_{\pi} \Hpp(\sfz(s), \sfp'(s)) +  \lambda_{\sfu,\sfp}(s)\, \sfp'(s) \ni 
(1{-} \lambda_{\sfu,\sfp}(s)) \, \serifsigma_\dev(s),   
\end{align}
\end{subequations}
where \eqref{diff-charact-BV-CR-1} and \eqref{diff-charact-BV-CR-3} need to be interpreted as \eqref{sense-displacements} and \eqref{sense-pl}, respectively.
    \end{proposition}
    \begin{remark}
    \upshape
    In comparison to  the differential characterization 
     for
$\BVZ$ solutions provided by system \eqref{diff-charact-BV-CL},
system  \eqref{diff-charact-BV-CR}  features \emph{two} parameters, instead of one. Both $\lambda_{\sfz}$ and 
 $ \lambda_{\sfu,\sfp}$ have the role of activating the viscous contributions to the damage flow rule, and to the displacement equation/plastic flow rule, respectively, in the jump regime (i.e.\ when $\sft'=0$).  In fact, the viscous terms in 
 \eqref{diff-charact-BV-CR-1} and \eqref{diff-charact-BV-CR-3}
 are modulated by the same parameter, which reflects the fact that  viscous behavior intervenes  for the  variables $\sfu$ and $\sfp$ equally (or, in other terms, that $\sfu$ and $\sfp$ relax to elastic equilibrium and rate-independent evolution at the same rate, faster than $\sfz$). 
    
    \end{remark}
    
    \par
    As in the case of Prop.\ \ref{diff-charact-BV-CL}, the proof
    of Prop.\ \ref{prop:diff-charact-BV-CR}
     will rely on a suitable chain-rule inequality, which we recall below.
 \begin{lemma}{\cite[Lemma 7.6]{Crismale-Rossi}}
 \label{l:ch-rule-pp-CR} 
     Along any \emph{enhanced admissible} parameterized curve 
  \[
    \begin{aligned}
    &
(\sft,\sfq) \in \eadm 0S0T{\Qpp}
\text{ s.t. }    \mathcal{M}^{0,0}_0(\sft,\sfq,\sft',\sfq')<+\infty \quad \aein\, (0,S) 
\end{aligned}
\]
we have that
  \begin{equation}
    \label{desired-chain-rule-pp-CR}
    \begin{aligned}
&
s \mapsto   \calE_{0}(\sft(s),\sfq(s))  \text{ is absolutely continuous on } [0,S] \text{ and there holds } \foraa\, s \in (0,S)
\\
&    
-\frac{\dd}{\dd s} \calE_{0}(\sft(s),\sfq(s)) +\partial_t  \calE_{0}(\sft(s),\sfq(s))\,  \sft'(s)  
 \\
  & = - \langle \rmD_z \calE_0(\sft(s),\sfq(s)), \sfz'(s) \rangle_{\Hs}  + \langle \serifsigma_\dev(s) \, |\, \sfp'(s) \rangle + \langle  \Diver\,\serifsigma(s) {+}  F(\sft(s)),  \sfu'(s) \rangle_{\BD}
  \\
  &
 \leq \mathcal{M}^{0,0}_0(\sft(s),\sfq(s),\sft'(s),\sfq'(s)).
\end{aligned}
    \end{equation}
 \end{lemma}
\noindent 
 For later use, we also record the following 
 consequence of the chain-rule inequality, cf.\ \cite[Proposition 7.7]{Crismale-Rossi}. 
 
\begin{corollary}
\label{rmk:upper}
An enhanced admissible parameterized   curve  $(\sft,\sfq)=(\sft, \sfu, \sfz, \sfp)\in \eadm 0S 0T{\Qpp}$ is  a 
$\BVZZZ$ solution
if and only if it 
satisfies 
one of the following \emph{equivalent} conditions:
\begin{enumerate}
\item the  energy-dissipation balance \eqref{EnDiss-CR} holds as the
inequality $\le$;
\item $(\sft,\sfq)$ 
fulfills \eqref{desired-chain-rule-pp-CR} as a chain of \emph{equalities}, i.e.\
\begin{equation}
\label{starting-identity}
\begin{aligned}
 & -\frac{\dd}{\dd s} \calE_{0}(\sft(s),\sfq(s)) +\partial_t  \calE_{0}(\sft(s),\sfq(s))\,  \sft'(s)  
 \\
  & = - \langle \rmD_z \calE_0(\sft(s),\sfq(s)), \sfz'(s) \rangle_{H^m}  + \langle \serifsigma_\dev(s) \, |\, \sfp'(s) \rangle + \langle  \Diver\,\serifsigma(s) {+}  F(\sft(s)),  \sfu'(s) \rangle_{\BD}
  \\
  &
 = \mathcal{M}^{0,0}_0(\sft(s),\sfq(s),\sft'(s),\sfq'(s)) \qquad \foraa\, s \in (0,S).
 \end{aligned}
\end{equation}
\end{enumerate}
\end{corollary}
\par
We are now in a position to carry out the
 \begin{proof}[Proof of Proposition \ref{prop:diff-charact-BV-CR}]
We exploit the characterization of $\BVZZZ$ solutions in terms of the chain of equalities \eqref{starting-identity}.
Now, we shall distinguish three cases:
\\
\CASE{$1$: $\sft'(s)>0$.}
Then, by the definition
\eqref{M0-CR}
 of $\calM^{0,0}_0$, from $\mathcal{M}^{0,0}_0(\sft(s),\sfq(s),\sft'(s),\sfq'(s))<\infty$ we infer
$\tilded_{L^2} ({-}\rmD_z \calE_0 (\sft(s),\sfq(s)),\partial\calR(0))=0$
and  
\[
 \calS_u \calE_0 (\sft(s),\sfq(s))=\calW_p\calE_0(\sft(s),\sfq(s))=0 .
 \]
 By \eqref{equivalence},
The latter property is equivalent  to  \eqref{eq:lambda-CL-1}, hence we find the validity of \eqref{diff-charact-BV-CR-1} with 
$ \lambda_{\sfu,\sfp}(s)=0$.  All in all, identity \eqref{starting-identity}  reduces to 
\begin{equation}
\label{starting-identity-2}
 \langle \serifsigma_\dev(s) \, |\, \sfp'(s) \rangle  - \Hpp(\sfz(s), \sfp'(s)) = \langle \rmD_z \calE_0(\sft(s),\sfq(s)), \sfz'(s) \rangle_{\Hs}  + \calR(\sfz'(s)).
\end{equation}
Now, since $\serifsigma(s) \in \mathcal{K}_{\sfz(s)}$ for a.a.\ $s\in (0,S)$, 
by \eqref{Prop3.9} 
the above left-hand-side  is negative.  On the other hand, from $\tilded_{L^2} ({-}\rmD_z \calE_0 (\sft(s),\sfq(s)),\partial\calR(0))=0$ we infer that 
$\calR(v) \geq  \langle {-}\rmD_z \calE_0(\sft(s),\sfq(s)), v \rangle_{\Hs} $ for every $v\in \Hs(\Omega)$. Therefore, the above right-hand side is positive. Hence, both sides are equal to zero. 
From 
\begin{equation}
\label{Hill-plast}
\Hpp(\sfz(s), \dot{\sfp}'(s)) =  \langle \serifsigma_\dev(s) \, |\, \sfp'(s) \rangle 
\end{equation}
 we infer (recall Lemma \ref{l:plastic-flow-rule})  that 
\eqref{diff-charact-BV-CR-3} holds with $ \lambda_{\sfu,\sfp}(s)=0$. Likewise, from 
\begin{equation}
\label{Hill-dam}
\calR(\sfz'(s)) = -  \langle \rmD_z \calE_0(\sft(s),\sfq(s)), \sfz'(s) \rangle_{\Hs},   
\end{equation}
 recalling 
\eqref{charact-calR} 
we deduce that $- \rmD_z \calE_0(\sft(s),\sfq(s) \in \partial \calR(\sfz'(s))$, i.e.\ the validity of \eqref{diff-charact-BV-CR-2} with $\lambda_{\sfz}(s) =0$.
\par
 Conversely, from 
 \eqref{diff-charact-BV-CR-2} with $\lambda_{\sfz}(s) =0$ and \eqref{diff-charact-BV-CR-3}  with $ \lambda_{\sfu,\sfp}(s)=1$ we deduce 
   \eqref{Hill-plast} and \eqref{Hill-dam},  respectively, hence \eqref{starting-identity-2} which, in this case, is equivalent to \eqref{starting-identity}.
   \par
   \noindent
\CASE{$2$:
$\sft'(s)=0$ and $s\in B^\circ$}  (with $B^\circ$ the set from \eqref{2906191815}).
 Then, from   $\mathcal{M}^{0,0}_0(\sft(s),\sfq(s),\sft'(s),\sfq'(s))<\infty$ we infer
that 
\begin{equation}
\label{0th-ingredient-CR}
\sfz'(s)=0,
\end{equation}
 and  \eqref{starting-identity} reduces to 
\begin{equation}
\label{key-identity-2}
\begin{aligned}
  & \langle \serifsigma_\dev(s) , \sfp'(s) \rangle_{L^2} + \langle  \Diver\,\serifsigma(s) {+}  F(\sft(s)),  \sfu'(s) \rangle_{H^1} 
  \\
  &   =  \calH(\sfz(s), \sfp'(s)) +    \DVito(\sfu'(s),\sfp'(s)) \, \DVito^*(\sft(s),\sfq(s))
 \\
 &
  = \calH(\sfz(s), \sfp'(s)) +     \sqrt{ \|\sfu'(s)\|_{H^1,\bbD}^2 {+} \|\sfp'(s)\|_{L^2(\Omega;\MD)}^2} \sqrt{
  \|\Diver\,\serifsigma(s) {+}  F(\sft(s))\|_{(H^1,\bbD)^*}^2 {+} 
   d_{L^2}(\serifsigma_\dev(s), \partial_\pi \mathcal{H}(z,0))^2 }\,.
  \end{aligned}
  \end{equation}
For \eqref{key-identity-2},  we have used that,
 on the left-hand side, the 
  duality pairing involving $\sfu'$ is between  $H^1(\Omega;\R^n)$ and $H^1(\Omega;\R^n)^*$ 
  since
the  admissible curve $(\sft,\sfq)$ enjoys the enhanced spatial regularity 
   $\sfu' \in H^1(\Omega;\R^n)$ on the set $B^\circ$. In turn, the right-hand side  of \eqref{key-identity-2} has been rewritten in view 
  of  \eqref{true-interpretation}. 
 Likewise,
 $\Hpp (\sfz(s), \sfp'(s)) =  \calH(\sfz(s), \sfp'(s))$ and 
  the stress-strain duality reduces to the scalar product in $L^2$ because $\sfp'(s) \in L^2(\Omega;\Mnn)$.
Let us now consider a measurable selection $s\mapsto \zeta(s) \in \partial_\pi \calH(\sfz(s),0) $ such that 
\[
 d_{L^2}(\serifsigma_\dev(s), \partial_\pi \calH(z,0))
  = \| \serifsigma_\dev(s) {-}\zeta(s)\|_{L^2(\Omega;\MD)} .
 \]
 Then, \eqref{key-identity-2} rewrites as
\[
\begin{aligned}
& 
 \langle \serifsigma_\dev(s) - \zeta(s) , \sfp'(s) \rangle_{L^2}    +  \langle  \Diver\,\serifsigma(s) {+}  F(\sft(s)),  \sfu'(s) \rangle_{H^1} 
 \\
 & \quad   - 
  \sqrt{ \|\sfu'(s)\|_{H^1,\bbD}^2 {+} \|\sfp'(s)\|_{L^2(\Omega;\MD)}^2} \sqrt{
  \|\Diver\,\serifsigma(s) {+}  F(\sft(s))\|_{(H^1,\bbD)^*}^2 {+} 
    \| \serifsigma_\dev(s) {-}\zeta(s)\|_{L^2(\Omega;\MD)}^2 }
   \\
   &  =    \calH(\sfz(s), \sfp'(s)) -\langle  \zeta(s) , \sfp'(s) \rangle_{L^2} .
   \end{aligned}
\]
While the left-hand side is  negative by Cauchy-Schwarz inequality,  the right-hand side is positive since $ \zeta(s) \in \partial_\pi \calH(\sfz(s),0)$, cf.\  
\eqref{charact-calR}. 
All in all, 
we conclude that both sides are equal to zero. Now, combining the fact that $ \zeta(s) \in \partial_\pi \calH(\sfz(s),0) $ with 
the identity
\[
 \calH(\sfz(s), \sfp'(s))  = \langle  \zeta(s) , \sfp'(s) \rangle_{L^2} 
\]
 and again resorting to \eqref{charact-calR}, 
we find that 
\begin{equation}
\label{first-ingredient-CR}
\zeta(s) \in \partial_{\pi} \calH(\sfz(s), \sfp'(s)).
\end{equation}
From the equality
\[
\begin{aligned}
&
\sqrt{ \|\sfu'(s)\|_{H^1,\bbD}^2 {+} \|\sfp'(s)\|_{L^2(\Omega;\MD)}^2} \sqrt{
  \|\Diver\,\serifsigma(s) {+}  F(\sft(s))\|_{(H^1,\bbD)^*}^2 {+} 
    \| \serifsigma_\dev(s) {-}\zeta(s)\|_{L^2(\Omega;\MD)}^2 }  
    \\
     & = \langle \serifsigma_\dev(s)-\zeta(s) , \sfp'(s) \rangle_{L^2} + \langle  \Diver\,\serifsigma(s) {+}  F(\sft(s)),  \sfu'(s) \rangle_{H^1} 
    \end{aligned}
\]
we infer that  
\begin{subequations}
\label{second-ingredient-CR}
\begin{align}
&
\label{second-ingredient-CR-1}
-\tilde\lambda(s)  \mathrm{div}\,\bbD \sig{\sfu'(s)}  =   \mathrm{Div}\,\serifsigma(s)  +    F(\sft(s))  && \text{in}\ H_\Dir^1(\Omega;\R^n)^* ,
\\
&
\label{second-ingredient-CR-2}
\tilde\lambda(s) \sfp'(s)  = \serifsigma_\dev(s) {-}\zeta(s)     && \aein\, \Omega\,,
\\
&
\label{second-ingredient-CR-3}
 \text{with } \tilde\lambda(s) =  \frac{ \sqrt{
  \|\Diver\,\serifsigma(s) {+}  F(\sft(s))\|_{(H^1,\bbD)^*}^2 {+} 
    \| \serifsigma_\dev(s) {-}\zeta(s)\|_{L^2(\Omega;\MD)}^2 }  }{\sqrt{ \|\sfu'(s)\|_{H^1,\bbD}^2 {+} \|\sfp'(s)\|_{L^2(\Omega;\MD)}^2}}\,.
\end{align}
\end{subequations}
Combining  \eqref{0th-ingredient-CR}, \eqref{first-ingredient-CR}, and  \eqref{second-ingredient-CR} we deduce the validity of system \eqref{diff-charact-BV-CR} 
with 
\[
 \lambda_{\sfz}(s)=1 \quad\text{ and } \quad  \lambda_{\sfu,\sfp}(s) = \frac{\tilde\lambda(s)}{1+\tilde\lambda(s)} \in (0,1).
 \]
 \par
 Conversely, 
 it can be easily checked that,   if
 $\sft'(s)=0$ and $s\in B^\circ$, 
   the validity of system \eqref{diff-charact-BV-CR} yields \eqref{key-identity-2},  hence
   \eqref{starting-identity}.  
 \par
 \noindent
 \CASE{$3$: $\sft'(s)=0$ and $s\notin B^\circ$.} Hence,  $\mathcal{D}^*(\sft(s), \sfq(s)) =0$, yielding
 \begin{equation*}
- \Diver\,\serifsigma(s) =   F(\sft(s)),
 \end{equation*}
 i.e.\ \eqref{diff-charact-BV-CR-1} with $  \lambda_{\sfu,\sfp}(s) =0$.
  Then,
   \eqref{starting-identity} reduces to 
\begin{equation}
\label{key-identity-3}
\begin{aligned}
&
   \langle {-}  \rmD_z \calE_0(\sft(s),\sfq(s)), \sfz'(s) \rangle_{\Hs} +   \langle \serifsigma_\dev(s) \, |\, \sfp'(s) \rangle   \\ & =  \calR(\sfz'(s)) +   \|\sfz'(s)\|_{L^2}\tilded_{L^2} ({-}\mathrm{D}_z \calE_0 (\sft(s),\sfq(s)),\partial\calR(0)) + \Hpp(\sfz(s), \sfp'(s))
   \end{aligned}
  \end{equation}
and, arguing in the same way as in the proof of Proposition \ref{diff-charact-BV-CL} 
we conclude the validity of 
  \eqref{diff-charact-BV-CR-2} and   \eqref{diff-charact-BV-CR-3}   with $  \lambda_{\sfu,\sfp}(s) =0$ and some  $  \lambda_{\sfz}(s) \geq 0$. 
  \par
Conversely,
it can be proved that \eqref{diff-charact-BV-CR-2} and   \eqref{diff-charact-BV-CR-3}   with $  \lambda_{\sfu,\sfp}(s) =0$ and some  $  \lambda_{\sfz}(s) \geq 0$ yield \eqref{key-identity-3}.
\par
With this, we conclude the proof.
\end{proof}
%

\section{A complete characterization of  $\BVZZZ$ solutions}
\label{sez:complete-charact} 

The main result of this section is the following theorem, whose proof adapts 
that of \cite[Prop.\ 5.5]{MRS14}.
\begin{theorem}
\label{thm:delusione}
 Let $(\sft,\sfq) \in \eadm 0S0T{\Qpp}$ be a 
 $\BVZZZ$ solution to the perfectly plastic rate-independent system for damage \eqref{RIS-intro}.
 Suppose that 
 $(\sft,\sfq)$ is \emph{non-degenerate}, namely that
 \begin{equation}
 \label{non-degeneracy}
 \sft'(s)+\|\sfu'(s)\|_{H^1(\Omega)}+\|\sfz'(s)\|_{\Hs(\Omega)}+\|\sfp'(s)\|_{L^2(\Omega)}>0  \qquad \foraa\, s \in (0,S).
 \end{equation}
   Set
\begin{equation*}
\mathfrak S:= \{ s \in [0,S]\, : \   \DVito^*(\sft(s),\sfq(s))=0\}.
\end{equation*}
Then, $\mathfrak S$ is either empty or it has the form $[s_*, S]$ for
some $s_*\in [0,S]$.
\begin{itemize}
\item[(a)] Assume $s_*>0$, then for $s\in [0,s_*)=[0,S]\setminus \mathfrak S$
we have $\sft(s) \equiv \sft(0)$
and $\sfz(s) \equiv  \sfz(0)$, whereas 
\begin{subequations}
\label{transition-layer}
\begin{align}
&
\label{transition-layer-1}
 \DVito^*(\sft(0),\sfq(0))>0
 \intertext{and 
$(\sfu,\sfp)$ is a solution to the
system}
& 
\label{transition-layer-2}
\begin{aligned}
& -\lambda_{\sfu,\sfp}(s)  \mathrm{div}\,\bbD \sig{\sfu'(s)} - (1{-} \lambda_{\sfu,\sfp}(s))   \mathrm{Div}\,\bbC(\sfz(0))\sfe(s)  =(1{-} \lambda_{\sfu,\sfp}(s))    F(\sft(s))\,,  
\\
&
(1{-} \lambda_{\sfu,\sfp}(s)) \, \partial_{\pi} \mathcal{H}(\sfz(0), \sfp'(s)) +  \lambda_{\sfu,\sfp}(s)\, \sfp'(s) \ni 
(1{-} \lambda_{\sfu,\sfp}(s)) \, (\bbC(\sfz(0))\sfe(s))_\dev. 
\end{aligned}
\end{align}
\end{subequations}
\item[(b)] 
Suppose that
 $s_*<S$.  Then,
$
 \DVito^*(\sft(s),\sfq(s)) \equiv 0 $  for every $ s\in
\mathfrak S=[s_*,S],$
and the curve $(\sft,\sfq)$ is a $\BV$ solution  to system
\eqref{RIS-intro}
in the sense of Definition \ref{def:BV-sol-CL}.
\end{itemize}
\end{theorem}
  \noindent
 Thus,  Theorem \ref{thm:delusione} provides a complete characterization of (non-degenerate) $\BVZZZ$ solutions. It
  asserts that, if a $\BVZZZ$ solution $(\sft,\sfq)$ starts from an unstable datum $q_0$ with $\DVito^*(0,q_0)>0$, then during an initial interval 
 the damage variable $\sfz$ is frozen and  the pair
$(\sfu,\sfp)$ evolves, possibly governed by viscosity in both variables. If it
 reaches,  at some 
 time $s^*$,  the state in which   elastic equilibrium ($\calS_u \calE_0 (\sft(s^*),\sfq(s^*))=0$)  and the plastic constraint
 ($\calW_p\calE_0(\sft(s^*),\sfq(s^*))=0$)
  are fulfilled, then it never leaves that state afterwards, and subsequently $(\sft,\sfq)$ behaves as a $\BVZ$ solution. 
\begin{proof}
\par
\noindent{\bf Step $1$:}
As in the proof of  \cite[Prop.\ 5.5]{MRS14}, we start by analyzing the behavior of a  $\BVZZZ$ solution 
$(\sft,\sfq) $ on an interval $(s_1,s_2) \subset [0,S] \setminus \mathfrak{S}$. Since  $ \DVito^*(\sft(s),\sfq(s))>0$ for all $s\in (s_1, s_2)$,   we read from \eqref{diff-charact-BV-CR-1} and \eqref{diff-charact-BV-CR-3} that $\lambda_{\sfu,\sfp}>0$  on  $ (s_1, s_2)$. Thus, \eqref{switch-BV-CR-1} yields $\sft' \equiv 0$ on 
$ (s_1, s_2)$, so that $\sft(s) \equiv \sft(s_1)$ for all $s\in [s_1, s_2]$. Furthermore,  \eqref{switch-BV-CR-2} 
gives $\lambda_{\sfz} \equiv 1$ on $ (s_1, s_2)$. Combining this with  \eqref{diff-charact-BV-CR-2} we gather that $\sfz' \equiv 0$ on  $ (s_1, s_2)$, so that $\sfz(s) \equiv \sfz(s_1)$ for all $s\in [s_1, s_2]$. 
From \eqref{non-degeneracy} we conclude that
\begin{subequations}
\label{non-degeneracy-conclusions}
\begin{equation}
\label{non-degeneracy-conclusions-1}
\sfu'(s) \neq 0 \text{ or } \sfp'(s) \neq 0 \qquad \foraa\, s \in (s_1,s_2),
\end{equation}
and, therefore,  from \eqref{diff-charact-BV-CR-1} and \eqref{diff-charact-BV-CR-3} we infer that 
$\lambda_{\sfu,\sfp}(s)<1$  for almost all $s\in (s_1,s_2)$. Hence, 
 the evolution of $(\sft,\sfq)$ in $(s_1,s_2)$ is characterized by 
property
\eqref{non-degeneracy-conclusions-1},
 joint with the previously found 
\begin{align}
\label{non-degeneracy-conclusions-2}
&
\sft(s) \equiv \sft(s_1), \ \ \sfz(s) \equiv \sfz(s_1)  && \text{ for all } s\in [s_1, s_2], 
\\
\label{non-degeneracy-conclusions-3}
& - \mathrm{div}\,\bbD \sig{\sfu'(s)}
-  \widehat{\lambda}(s)    \mathrm{Div}\,\serifsigma(s)  = 
 \widehat{\lambda}(s)    F(\sft(s_1)) && \foraa\, s \in (s_1,s_2),  
\\
&
\label{non-degeneracy-conclusions-4}
 \widehat{\lambda}(s)    \zeta(s) +   \sfp'(s) \ni 
 \widehat{\lambda}(s)    \serifsigma_\dev(s)  && \foraa\, s \in (s_1,s_2)\,,
\\
&
 \text{with }   
 \widehat{\lambda}    := \tfrac{1{-}\lambda_{\sfu,\sfp}}{\lambda_{\sfu,\sfp}} = 
\frac{\sqrt{ \|\sfu'\|_{H^1,\bbD}^2 {+} \|\sfp'\|_{L^2(\Omega;\MD)}^2}}{\DVito^*(\sft,\sfq)}
 \nonumber
\end{align}
\end{subequations}
(observe that $\widehat{\lambda} = \tilde{\lambda}^{-1}$ with $\tilde\lambda$ from, \eqref{second-ingredient-CR}). 
\noindent
In turn,  it can be easily checked that, for a \emph{given} function $\widehat{\lambda} : (s_1,s_2) \to [0,+\infty) $,  the Cauchy problem for  system \eqref{non-degeneracy-conclusions-2}---\eqref{non-degeneracy-conclusions-4} does
admit a unique solution.  
\par
\noindent
\textbf{Step $2$:} 
Since the function $[0,S] \ni s\mapsto  \DVito^*(\sft(s),\sfq(s)) $ is lower semicontinuous by 
\cite[Lemma 7.8]{Crismale-Rossi},  
the set $\mathfrak{S}$ is closed, hence its complement $[0,S]\setminus \mathfrak{S}$ is relatively open, and thus it  is the finite or countable union of disjoint intervals. Its connected components are  of the form $(s_1,S]$, 
$(s_2,s_3)$, $[0,s_4)$ or $[0,S]$ (if $\mathfrak{S} = \emptyset$). By the lower semicontinuity of $s\mapsto  \DVito^*(\sft(s),\sfq(s)) $, it is immediate to check that $s_j \in \mathfrak{S}$ for all  $j \in \{1,\ldots, 4\}$. 
\par
 Now, we aim to show that connected components of the type  $(s_1,S]$ and 
$(s_2,s_3)$
cannot occur. To this end, 
let us study the properties of  the $\BVZZZ$ solution 
$(\sft,\sfq) $ on an interval of the type $(s_1,S]$ or $(s_2,s_3)$, with $s_1,\, s_2 \in \mathfrak{S}$. Since
$\DVito^*(\sft(s_j),\sfq(s_j)) = 0$, we have
that
$\calS_u \calE_0(\sft(s_j), \sfq(s_j))= \mathcal{W}_p \calE_0(\sft(s_j),\sfq(s_j))=0$ for $j=1,2$. 
As shown in Step $1$, the evolution on the intervals $(s_1,S)$ and $(s_2,s_3)$ is characterized by 
\eqref{non-degeneracy-conclusions}.   Recall that system \eqref{non-degeneracy-conclusions-2}--\eqref{non-degeneracy-conclusions-4} admits a unique solution.  Now, 
%
 since 
$\calS_u \calE_0(\sft(s_j), \sfq(s_j))= \mathcal{W}_p \calE_0(\sft(s_j),\sfq(s_j))=0$ for $j=1,2$, it is immediate to check that the constant  functions 
$\widehat{\sfu}(s) \equiv \sfu(s_j)$ and $\widehat{\sfp}(s) \equiv \sfp(s_j)$ provide the unique solution to 
 \eqref{non-degeneracy-conclusions-2}--\eqref{non-degeneracy-conclusions-4}. Thus, we conclude that $\BVZZZ$ solution 
$(\sft,\sfq) $ on an interval of the type $(s_1,S]$ or $(s_2,s_3)$ must be constant, which is a contradiction to \eqref{non-degeneracy-conclusions-1}.  
%
%
\par
 Therefore, $[0,S]\setminus \mathfrak{S}$  does not possess connected components of the form $(s_1,S]$ or 
 $(s_2,s_3)$. Hence, either  $\mathfrak{S}=\emptyset$, or $\mathfrak{S} = [s_*,S]$ for some $s_*>0$.  In the latter case,
 the calculations from Step $1$ show that 
  on $[0,S]\setminus \mathfrak{S}  = [0,s_*)$ the evolution of $(\sft,\sfq)$ is characterized by \eqref{transition-layer}.
  \par
\noindent
\textbf{Step $3$:}  Suppose that  $s_*<S$. Clearly, $
 \DVito^*(\sft(s),\sfq(s)) \equiv 0 $  for every $ s\in
\mathfrak S=[s_*,S]$. Hence it satisfies system \eqref{diff-charact-BV-CR}   with $\lambda_{\mathsf{u}, \mathsf{p}} (s) \equiv 0 $ for every $s\in [s_*,S]$, which coincides with system \eqref{differential-chact-CL}. 
This concludes the proof.
\end{proof}

\section{Vanishing-hardening limit of $\BV$ solutions}
\label{s:4}
%
%
 In this section we 
carry out the asymptotic analysis as  the hardening parameter $\mu$ tends to 0 for the 
$\BV $ solutions to system \eqref{RIS-hard-intro}, both in the    \emph{single-rate}  case
(i.e., for $\BV_{0}^{\mu,\nu}$ solutions, with $0<\nu\leq \mu$),
and in the \emph{multi-rate} case (i.e, for
$\BV_{0}^{\mu,0}$ solutions). Indeed, we first address the latter case in Section \ref{sec:lim-double} ahead,
while the former will be sketched in Sec.\ \ref{sec:lim-single}. 
For both analyses, 
we will resort to some technical results collected in the Appendix. 
%
%
%
%
%
\subsection{Vanishing-hardening analysis for \emph{multi-rate} solutions}\label{sec:lim-double}
As recalled in the Introduction,
$\BV$ solutions to the \emph{multi-rate} system with hardening have been constructed   in \cite{Crismale-Rossi} (cf.\ Theorem 6.13 therein) by passing to the limit
in the
 (reparameterized) version of  \eqref{RD-intro} as  the viscosity parameter 
 $\eps $ tends to $ 0$ simultaneously with  the rate parameter $\nu \down 0$,  while the hardening parameter $\mu>0$  stayed fixed.    The  solutions accordingly obtained,  hereafter referred to as $\BVA{\mu}$ solutions,  thus account for multiple rates in the system with hardening. In particular, like for $\BVZZZ$ solutions,  the way in which viscous behavior  in   $u$, $z$, and  $p$
 manifests itself in the jump regime  reflects the fact that 
 the convergence of $u$ and $p$ to 
   elastic equilibrium and rate-independent evolution 
  has    occurred  at a \emph{faster rate} (as $\nu\down 0$) than that for $z$, cf.\ Remark \ref{rmk:comment-below} \ ahead.  
\par
In order to recall the definition of $\BVA{\mu}$ solutions for fixed $\mu>0$, we need to   introduce  the related  vanishing-viscosity contact potential 
\begin{equation*}
\begin{aligned}
&
\Mlinamezero: [0,T]  \times \Qha \times [0,+\infty) \times \Qha \to [0,+\infty],
\\
&
\Mlizero tq{t'}{q'} := 
 \calR(z') + \calH(z,p') + \Mliredzero t u z p {t'}{u'}{z'}{p'}
\end{aligned}
\end{equation*}
where
\begin{subequations}
\label{Mli-def}
\begin{align}
&
\label{Mli>0'}
& 
\text{if } t'>0,  \quad &
 \Mlirednamezero(t,q,t',q') := 
\begin{cases}
 0 &\text{if } 
\begin{cases}
-\mathrm{D}_u \calE_\mu (t,q)=0 ,  \\
-\mathrm{D}_z  \calE_\mu  (t,q) \in \partial\calR(0) , \ \text{and} \\
-\mathrm{D}_p  \calE_\mu (t,q) \in \partial_\pi \calH(z,0),
\end{cases}
 \\
 +\infty & \text{otherwise,}
\end{cases}
\\ &&
\label{Mlizero=0}
\text{if } t'=0,  \quad &
 \Mlirednamezero(t,q,t',q') : =
\begin{cases}
    \DVito(u', p') \, \DVitosred(t,q)   & \text{if } z'=0,
\\
 \|z'\|_{L^2}\, \tilded_{L^2} ({-}\mathrm{D}_z \calE_{ \mu } (t,q),\partial\calR(0)) & \text{if } \DVitosred(t,q)=0
 \\
 &
  \text{ and } \tilded_{L^2}  ({-}\mathrm{D}_z \calE_\mu (t,q),\partial\calR(0))<+\infty,
\\
+ \infty & \text{otherwise}.
\end{cases}
\end{align}
In \eqref{Mli-def} we have  employed the notation
\begin{equation}\label{3006191243}
\begin{aligned}
&
  \DVito(u', p'): =   \sqrt{ \|u'\|_{H^1,\bbD}^2 {+} \|p'\|_{L^2(\Omega;\MD)}^2}, 
\\
&   \DVitosred(t,q)   :=  \sqrt{\|{-}\mathrm{D}_u  \calE_\mu  (t,q)\|^2_{( H^1, \bbD )^*}{+}  d_{L^2}({-}\mathrm{D}_p \calE_\mu (t,q),\partial_\pi \calH(z,0))^2}. 
\end{aligned}
\end{equation} 
\end{subequations}


 We are now in a position to  recall the notion of solution to system  \eqref{RIS-hard-intro} from
\cite[Def. 6.10]{Crismale-Rossi}. Observe that it involves (reparameterized) curves that are absolutely continuous on the \emph{whole}  interval $[0,S]$, with values in  $\Qha$. 
  \begin{definition}
   \label{def:multi-rate-BV-solution-hardening}
   We call a parameterized curve $(\sft,\sfq) = (\sft,\sfu,\sfz,\sfp) \in \AC ([0,S]; [0,T]{\times} \Qha)$
   a  \emph{(parameterized) Balanced Viscosity} solution to the \emph{multi-rate}   system with hardening 
   \eqref{RIS-hard-intro} 
   (a $\BVA{\mu}$ solution, for short), 
    if  $\sft \colon [0,S] \to [0,T]$ is nondecreasing and $(\sft, \sfq)$  fulfills for all $0\leq s \leq S$  the energy-dissipation balance
   \begin{equation}
\label{ED-balance-zero}
   \calE_{ \mu }(\sft(s),\sfq(s)) + \int_{0}^{s}
   \Mlizero{\sft(\tau)}{\sfq(\tau)}{\sft'(\tau)}{\sfq'(\tau)} \dd \tau = \calE_{ \mu } (\sft(0),\sfq(0)) +\int_{0}^{s} \partial_t \calE_{ \mu } (\sft(\tau), \sfq(\tau)) \, \sft'(\tau) \dd \tau .
\end{equation}
We say that $(\sft,\sfq)$ is \emph{non-degenerate} if it fulfills \eqref{non-degeneracy}. 
   \end{definition}
   \begin{remark}
   \label{rmk:comment-below}
   \upshape
   In \cite[Prop.\ 6.11]{Crismale-Rossi} a  characterization was provided  for  $\BVA{\mu}$ solutions in terms of a subdifferential system that features two positive parameters $\lambda_{\sfu,\sfp}$ and $\lambda_{\sfz}$ activating viscous terms in the displacement equation \& plastic flow rule, and in the damage flow rule, respectively.  We refrain from recalling that system because it is completely analogous to
\eqref{diff-charact-BV-CR}
(with the same \emph{switching conditions}   \eqref{switch-BV-CR}).
\par
 Accordingly, repeating the very same arguments as in the proof of Theorem \ref{thm:delusione}, it is  possible to provide an additional characterization of $\BVA{\mu}$ solutions  completely analogous to that from the latter result. In particular,  if a $\BVA{\mu}$ solution $(\sft,\sfq)$ originates from an unstable datum $q_0$ with $\DVitosred(0,q_0)>0$, then, during an initial interval 
 the damage variable $\sfz$ is frozen and  the pair
$(\sfu,\sfp)$ evolves, possibly  in a viscous way. If $(\sft,\sfq)$
 reaches,  at some 
 time $s^*$,  the state in which   elastic equilibrium ($ -\mathrm{D}_u \calE_\mu (\sft,\sfq)=0$)  and stability 
($-\mathrm{D}_p  \calE_\mu (\sft,\sfq) \in \partial_\pi \calH(z,0)$)
  are fulfilled, then it never leaves that state afterwards, and subsequently behaves as a
   Balanced Viscosity solution 
  with viscous behavior in the variable $\sfz$, only (namely, the counterpart, for the system with hardening,  of the $\BVZ$ concept). 
 \par
 We mention in advance that the analogue of  Theorem \ref{thm:delusione} does not hold, instead, for $\BV_0^{\mu,\nu}$ solutions. 
   \end{remark}


We now consider a vanishing sequence $(\mu_k)_k$ and set $\mu=\mu_k$.
By \cite[Theorem 6.13]{Crismale-Rossi},
under the assumptions of Section~\ref{s:2} and \eqref{EL-initial-CL} (cf.\ also Remark \ref{rmk:more-general-init-data} below), 
for any fixed $k$ there exists a parameterized Balanced Viscosity solution
$(\sft_k,\sfq_k)$ in the sense of the previous definition.
Moreover,  for the sequence
$(\sft_k,\sfq_k)_k$ 
 we  may assume   the validity of the following a priori estimates:
\begin{equation}\label{assum:bound}
\begin{aligned}
\exists\, C>0 \ \ \forall\, k \in \N \, \ \foraa\, s \in (0,S)\, : \qquad &  {\sft}_k'(s) {+}  \| \sfu_k'(s)\|_{W^{1,1}(\Omega)} {+} \|\sfz_k'(s)\|_{\Hs(\Omega)} {+} \|\sfp_k'(s)\|_{L^1(\Omega)} {+} \sqrt{\mu_k} \, \|\sfp_k'(s)\|_{L^2(\Omega)} 
\\
& \quad    
 {+} \|\sfe_k'(s)\|_{L^2(\Omega)}  {+}     \DVito(\sfu_k'(s), \sfp_k'(s)) \, \DVito^{*,\mu_k}(\sft_k(s),\sfq_k(s))
\leq C.
 \end{aligned}
\end{equation}
Indeed, 
the existence of a sequence $(\sft_k,\sfq_k)_k$ 
 enjoying the bounds \eqref{assum:bound} follows by time-discretization, 
  cf.\  \cite[Prop.\ 4.4]{Crismale-Rossi}, and by a reparameterization argument. 
  Up to a further time reparameterization we may also assume that the solutions are non-degenerate, cf.\ \eqref{non-degeneracy} and \cite[Remark 6.9]{Crismale-Rossi}.
%
\begin{theorem}\label{teo:exparBVsol}
 Let $(\mu_k)_k$ be a vanishing sequence
and $(\sft_k,\sfq_k)_k$ be a  sequence of 
$\BVA{\mu_k}$ solutions to system \eqref{RIS-hard-intro}, 
such that estimate \eqref{assum:bound} holds.
\par
Then,
there exist a (not relabeled) subsequence 
and  a 
curve  $ (\sft, \sfq)=(\sft, \sfu, \sfz, \sfp) \in  \eadm 0S0T{\Qpp}$ such that   
\begin{enumerate}
\item for all $s\in[0,S]$ the following convergences hold as $k\to+\infty$
\begin{equation}\label{weak-converg-pp}
\begin{split}
&\sft_{k}(s)\to\sft(s) , \quad  \sfu_{k}(s) \weaksto \sfu(s) \text{ in } \BD(\Omega) , 
 \quad  \sfz_{k}(s) \weakto \sfz(s) \text{ in } \Hs(\Omega) ,  \\
& \sfe_k(s)  \weakto \sfe(s) \text{ in } \Lnn,\quad  \sfp_{k}(s) \weaksto \sfp(s) \text{ in } \MbD;
 \end{split}
\end{equation}
\item
 there exists $\overline{C}>0$ such that 
for a.e.\ $s\in (0, S)$ there holds
\begin{equation}\label{2906191858}
\begin{aligned}
&
{\sft}'(s) {+}  \| \sfu'(s)\|_{\BD(\Omega)} {+} \|\sfz'(s)\|_{\Hs(\Omega)} {+} \|\sfp'(s)\|_{\MbD} 
\\
&
 {+} \|\sfe'(s)\|_{\Lnn} 
  {+} \DVito(\sfu'(s), \sfp'(s))\, \DVito^* (\sft(s), \sfq(s))\leq \overline{C};
 \end{aligned}
\end{equation}
 \item
 $(\sft,\sfq)$ is a Balanced Viscosity solution for the perfectly plastic damage system \eqref{RIS-intro}
  in the sense of Definition \ref{def:BV-sol-CR}. 
 \end{enumerate}
\end{theorem}

\begin{remark}
\label{rmk:more-general-init-data}
\upshape
The validity of Theorem \ref{teo:exparBVsol} extends to sequences $(\sft_k,\sfq_k)_k$ originating from initial data  
%
 $q_0^k = (u_0^k,z_0^k,p_0^k)_k$ fulfilling  
\[
\|u_0^k\|_{H_\Dir^1(\Omega)} + \|z_0^k\|_{\Hs(\Omega)}+  \| W(z_0^k) \|_{L^1(\Omega)}  + \|p_0^k\|_{L^2(\Omega)}  + \| \rmD_q  \enen{0} (0,q_0^k) \|_{L^2(\Omega)}   \le C
\]
with $C$ independent of $k$. 

In particular, the last condition yields  $\mu_k \, p_0^k \to 0$ in $L^2(\Omega;\MD)$, as needed  for \eqref{convergence-of-energies} below. 

\end{remark}
\begin{proof}  The proof is divided in two steps.
\paragraph{\bf Step $1$: Compactness.} 
For later use, we observe that, due to estimate  \eqref{assum:bound},
\begin{equation}
\label{they-stay-in-ball}
\exists\, R>0 \ \ \forall\, k \in \N \ \ \forall\, s \in [0,S]\, : \quad 
\|\sfu_k(s)\|_{\BD(\Omega)} + \|\sfz_k(s)\|_{\Hs(\Omega)}  + \| \sfp_k(s)\|_{\mathrm{M}_{\mathrm{b}}(\Omega{\cup}\Gdir)} \leq R.
\end{equation}
By the assumptions on initial data and external loading and  by \eqref{assum:bound}, we have 
$\sup_{s\in [0,S]}\calE_{ \mu_k }(\sft_k(s),\sfq_k(s))\le C$   for a constant independent of $k$. 
In particular, this implies that, as $k\to\infty$, 
\begin{equation}\label{2307191104}
{\mu_k} \, \sfp_k(s) \to 0\quad\text{in }L^2(\Omega;\MD)
\text{ for every $s \in [0,S]$.}
\end{equation}
\par
In view of   \eqref{assum:bound},
we find a (not relabeled)  subsequence and a Lipschitz  curve
$(\sft,\sfq) \in $ $  W^{1,\infty} ([0,S];[0,T]{\times}\Qpp)$ such that the following convergences hold as $k\to\infty$
  \begin{subequations}\label{2906191911}
\begin{align}
& \sft_{k}\weaksto\sft  \quad \text{in } W^{1,\infty}(0,S;[0,T]) , && \sfu_{k}\weaksto\sfu \quad \text{in } W^{1,\infty}(0,S;\BD(\Omega)) ,
 \\
&  \sfz_{k}\weaksto\sfz \quad \text{in } W^{1,\infty}(0,S;\Hs(\Omega)) , &&
\\
&\sfe_{k}\weaksto\sfe \quad \text{in } W^{1,\infty}(0,S;\Lnn), &&
\sfp_{k}\weaksto\sfp \quad \text{in } W^{1,\infty}(0,S;\MbD) ,
\end{align}
\end{subequations}
where $\sfe = \sig{\sfu{+}w(\sft)} - \sfp$. 
Furthermore, 
an argument based on the Ascoli-Arzel\`a theorem  (cf.\ \cite[Prop.\ 3.3.1]{AGS08}) also yields
\begin{subequations}
\label{Co-cvg}
\begin{align}
&
\sfu_{k} \to \sfu   && \text{in } \mathrm{C}^0([0,S]; \BD(\Omega)_{\mathrm{w}^*}), \label{2204071}
\\
&
\sfe_{k} \to \sfe   && \text{in }   \mathrm{C}^0([0,S]; L^2(\Omega;\Mnn)_{\mathrm{w}}),
\\
&
\sfz_{k} \to \sfz   && \text{in }   \mathrm{C}^0([0,S]; \Hs(\Omega)_{\mathrm{w}}),
\\
&
\sfp_{k}\to\sfp  &&  \text{in } \mathrm{C}^0([0,S]; \MbD_{\mathrm{w}^*}).  \label{2204072}
\end{align}
\end{subequations}
Indeed, 
convergences \eqref{2204071} and \eqref{2204072}
are to be intended in the spaces
$ \mathrm{C}^0([0,S]; (\BD(\Omega), d_{\BD,\mathrm{weak}^*}))$ and in 
$ \mathrm{C}^0([0,S]; (\MbD, d_{\mathrm{M}_\mathrm{b},\mathrm{weak}^*}))$, 
where $d_{\BD,\mathrm{weak}^*}$ and $d_{\mathrm{M}_\mathrm{b},\mathrm{weak}^*}$ metrize the weak$^*$ topologies of $\BD(\Omega)$ and $\MbD$, respectively, on the balls of radius $R$ that contain $(\sfu_k)_k$ and $(\sfp_k)_k$, resp.\ (cf.\ 
\eqref{they-stay-in-ball}; here we use that $\BD(\Omega)$ is the dual of a separable space).
The second and third convergences have an analogous meaning. 
Hence, \eqref{weak-converg-pp} follows. 
\par
With the very same arguments as in the proof of \cite[Prop.\ 7.9]{Crismale-Rossi}, 
based on the estimate 
\[
\exists\, C>0 \  \forall\, k \in \N \ \foraa\, s\in (0,S)\, : \   \DVito(\sfu_k'(s), \sfp_k'(s)) \, \DVito^{*,\mu_k}(\sft_k(s),\sfq_k(s))\leq C,
\]
we also prove the enhanced regularity
$(\sft,\sfq) \in  \eadm 0S0T{\Qpp}$.  
\par
\paragraph{\bf Step $2$:  energy-dissipation upper estimate.} 
By  Corollary  \ref{rmk:upper}, it is sufficient to show
 that the pair $(\sft,\sfq)$ complies with 
 the energy-dissipation inequality
\begin{equation*}
\calE_0(\sft(s),\sfq(s)) + \int_{0}^{s} \calM^{0,0}_0(\sft(r),\sfq(r),\sft'(r),\sfq'(r))
 \dd r \le  \calE_0(\sft(0),\sfq(0)) + \int_{0}^{s} \partial_t \calE_0(\sft(\tau), \sfq(\tau)) \dd \tau
   \end{equation*}
   for every $s\in [0,S]$. 
We start from \eqref{ED-balance-zero} for the solution $(\sft_k,\sfq_k)$ with $\mu=\mu_k$.
It is straightforward to see that
\begin{equation}
\label{convergence-of-energies}
\calE_0(\sft(s),\sfq(s))\le \liminf_{k\to+\infty}\calE_{ \mu_k }(\sft(s),\sfq(s)) , \qquad
\calE_0(\sft(0),\sfq(0))=\lim_{k\to+\infty}\calE_{ \mu_k }(\sft(0),\sfq(0)) ,
\end{equation}
and
\[
\int_{0}^{s} \partial_t \calE_0(\sft(\tau), \sfq(\tau)) \dd \tau=\lim_{k\to+\infty}
\int_{0}^{s} \partial_t \calE_{\mu_k}(\sft_k(\tau), \sfq_k(\tau)) \dd \tau .
\]
It remains to show that
\begin{equation} \label{eq:sci-giu}
 \int_{0}^{s} \calM^{0,0}_0(\sft(r),\sfq(r),\sft'(r),\sfq'(r)) \dd r \leq 
  \liminf_{k\to+\infty}
  \int_{0}^{s}  \calM_{0}^{\mu_k,0}(\sft_{k}(r), \sfq_{k}(r), \sft_{k}'(r), \sfq'_{k}(r)) \dd r   .
\end{equation} 
In fact, it will be sufficient to obtain the above estimate only for the reduced functionals 
$\calM_{0,\mathrm{red}}^{0,0}$ and $ \calM_{0,\mathrm{red}}^{\mu_k,0}$. 
 In view of \eqref{M0-CR} we distinguish two cases.
Let $A: = \{ s\in [0,S]\,: \ \sft'(s)>0\}$. 
\paragraph{\bf Case $\sft'>0$.} 
We prove that the function  
$s \mapsto \oMliredname(\sft(s), \sfq(s), \sft'(s), \sfq'(s))$ is finite for a.a.\ $s\in A$.
We  apply  Lemma \ref{lemma:giu} ahead with the choices $\mathsf{f}_k: = \sft_k$, $\mathsf{f}: = \sft$. Thus, we conclude that  
 for a.a.\ $s\in A$ there is a subsequence $(k_j)_{j}$
and, for every $j$, there is $s_{k_j}\in(0,S)$ such that
$|s_{k_j}{-}s|<\frac1j$ and $\sft'_{k_j}(s_{k_j})>0$.
In particular, $\calM_{0}^{\mu_k,0}(\sft_{k_j}(s_{k_j}), \sfq_{k_j}(s_{k_j}), \sft_{k_j}'(s_{k_j}), \sfq'_{k_j}(s_{k_j}))=0$.
By \eqref{M0-CR},  \eqref{2307191104},  convergences \eqref{Co-cvg}, 
and Lemma \ref{le:2307191000} ahead, we obtain
\begin{equation*}
\begin{aligned}
 \calS_u \calE_0(\sft(s), \sfq(s)) =0 , \qquad 
-\mathrm{D}_z \calE_0 (\sft(s),\sfq(s)) \in \partial\calR(0) , \qquad 
\mathcal{W}_p \calE_0(\sft(s),\sfq(s))=0
\quad
 \text{for a.a.}\ s\in A,
\end{aligned}
\end{equation*} 
which is equivalent to state that $\oMliredname(\sft(s), \sfq(s), \sft'(s), \sfq'(s))=0$.
Hence,
 we obviously have  the \emph{pointwise} estimate 
\begin{equation}
\label{pointwise-endiss}
  \oMliredname(\sft(s),\sfq(s),\sft'(s),\sfq'(s))  \leq 
  \liminf_{k\to+\infty}
 \calM_{0,\mathrm{red}}^{\mu_k,0}(\sft_{k}(s), \sfq_{k}(s), \sft_{k}'(s), \sfq'_{k}(s)) \qquad \foraa\, s \in A.
 \end{equation}
 
\paragraph{\bf Case $\sft'=0$.}  By virtue of Lemma \ref{le:2307191000},  
we are in a position to  apply Lemma \ref{l:MIE-ROS} below
in the context of the space $\boldsymbol{Q}: = \Qpp$, with $\boldsymbol{S}$ the ball of radius $R$ from \eqref{they-stay-in-ball}, 
 to the functionals 
$\mathscr{M}_k: = \calM_{0,\mathrm{red}}^{\mu_k,0}$ and $\mathscr{M}_0: = \calM_{0,\mathrm{red}}^{0,0}$,
 with $\calM_{0,\mathrm{red}}^{\mu_k,0}$  extended to  $\R \times (\Qpp{\setminus} \Qha) \times \R \times (\Qpp{\setminus} \Qha)$ by setting  $\calM_{0,\mathrm{red}}^{\mu_k,0}(t,q,t',q')=+\infty$ when $q$ or $q' \in \Qpp{\setminus} \Qha$. 
We may then observe that 
the $\Gamma$-$\liminf$ estimate \eqref{G-liminf} in Lemma  \ref{l:MIE-ROS}   follows from Lemma \ref{le:2307191000}. 
Thus, we conclude
\[
\int_{(0,s) \setminus A}  \calM_{0,\mathrm{red}}^{0,0}(\sft(r),\sfq(r), 0 ,\sfq'(r)) \dd r   \leq 
  \liminf_{k\to+\infty}
\int_{(0,s) \setminus A} \calM_{0,\mathrm{red}}^{\mu_k,0}(\sft_{k}(r), \sfq_{k}(r), \sft_{k}'(r), \sfq'_{k}(r)) \dd r.
\]
All in all, \eqref{eq:sci-giu} follows. 
\par
This finishes the proof.  
\end{proof}

\subsection{Vanishing-hardening analysis for  \emph{single-rate} solutions}\label{sec:lim-single} $\BV$ solutions to the \emph{single-rate} system with hardening have been  obtained in \cite[Section~6.1]{Crismale-Rossi}  by 
 performing the asymptotic analysis of the reparameterized energy-dissipation balance
\eqref{param-endissbal} as  the viscosity parameter $\eps$ tends to  $0$, while keeping the hardening and the rate parameters $\mu$ and  $\nu$ fixed. That is why we  refer to them as   $\BVB{\mu,\nu}$ solutions to the system with hardening. 
Their definition involves the 
corresponding
  vanishing-viscosity contact potential   $\Mliname:  [0,T]  \times \Qha \times [0,+\infty) \times \Qha \to [0,+\infty] $ given by
   \begin{subequations}
\label{Mli-def'}
\begin{align}
&
\Mli tq{t'}{q'}
: = \calR(z') + \calH(z,p') +  \Mliredname(t,q,t',q') , \quad \text{where} 
\nonumber
\\
&
\label{Mli>0}
\begin{aligned}
& 
\text{if } t'>0,  \quad 
 \Mliredname(t,q,t',q') := 
\begin{cases}
 0 &\text{if } 
\begin{cases}
-\mathrm{D}_u \calE_\mu (t,q)=0 ,  \\
-\mathrm{D}_z  \calE_\mu  (t,q) \in \partial\calR(0) , \ \text{and} \\
-\mathrm{D}_p  \calE_\mu (t,q) \in \partial_\pi \calH(z,0),
\end{cases}
 \\
 +\infty & \text{otherwise,}
\end{cases}
\end{aligned}
\\
& 
\label{Mli=0}
\begin{aligned}
&
\text{if } t'=0,  \quad 
 \Mliredname(t,q,0,q') : = \DVito_\nu(q') \, \DVitos(t,q).
\end{aligned}
\end{align}
\end{subequations}
 For better readability, we also recall that 
\[
\begin{aligned}
  & 
   \DVito_\nu({q}'): = 
    \sqrt{ \nu \| {u}'(t)\|^2_{ H^1, \bbD }   {+} 
\|{z}'(t)\|_{L^2}^2
{+} \nu\|{p}'(t)\|_{L^2}^2} ,
\\
&
 \DVitos({t},{q}): =   \sqrt{
\frac1{\nu}\, \|{-}\mathrm{D}_u \calE_\mu ({t},{q} )\|^2_{( H^1, \bbD )^*} 
+ \tilded_{L^2} ({-}\mathrm{D}_z \calE_\mu  ({t},{q}) ,\partial\calR(0))^2
+ \frac1\nu \, \dLtwo ({-}\mathrm{D}_p \calE_\mu  ({t},{q}) ,\partial_\pi \calH( {z} ,0))^2
}.
\end{aligned}
\]
 It is worthwhile to remark that the \emph{reduced} functional  $ \Mliredname$, at $t'=0$, \emph{simultaneously} encompasses viscosity for the three variables $u$, $z$, and $p$. Instead,
  its counterpart 
 $ \Mlirednamezero$ for $\BVA{\mu}$ solutions features, in the jump regime $t'=0$,  a viscous contribution in the variables $(u,p)$ when $z'=0$, and viscosity in $z$ when  $ \DVitosred(t,q)=0$, i.e. when $u$ is at elastic equilibrium and $p$ is locally stable.
 \par
 We are now in a position to recall 
 the notion of $\BVB{\mu,\nu}$ solution,  cf.\ \cite[Definition~6.2]{Crismale-Rossi}. 
 \begin{definition}
   \label{def:single-rate-BV-solution-hardening}
   We call a parameterized curve $(\sft,\sfq) = (\sft,\sfu,\sfz,\sfp) \in \AC ([0,S]; [0,T]{\times} \Qha)$
   a  \emph{(parameterized) Balanced Viscosity} solution to the \emph{single-rate}   system with hardening 
    \eqref{RIS-hard-intro} 
   (a $\BVB{\mu,\nu}$ solution, for short), 
    if  $\sft \colon [0,S] \to [0,T]$ is nondecreasing and $(\sft, \sfq)$  fulfills for all $0\leq s \leq S$  the energy-dissipation balance
   \begin{equation}
\label{ED-balance-zero-single}
   \calE_{ \mu }(\sft(s),\sfq(s)) + \int_{0}^{s}
   \Mli{\sft(\tau)}{\sfq(\tau)}{\sft'(\tau)}{\sfq'(\tau)} \dd \tau = \calE_{ \mu } (\sft(0),\sfq(0)) +\int_{0}^{s} \partial_t \calE_{ \mu } (\sft(\tau), \sfq(\tau)) \, \sft'(\tau) \dd \tau .
\end{equation}
We say that $(\sft,\sfq)$ is \emph{non-degenerate} if it fulfills \eqref{non-degeneracy}. 
   \end{definition}
 \par
 Let us now address the asymptotic analysis of the above solutions for 
a   vanishing sequence $(\mu_k)_k$.
   As mentioned in the Introduction, in the construction of $\BVB{\mu,\nu}$ solutions 
   the rate parameter is always supposed  smaller than the hardening parameter, 
   which forces us to also consider  a sequence $(\nu_k)_k$ such that 
    $\nu_k\leq \mu_k$ for all $k\in \N$, so that  $\nu_k\to 0$ as well.  In fact, the
   technical
     condition $\nu_k \leq \mu_k$ comes into play in the proof of 
      \cite[Prop.\ 4.4]{Crismale-Rossi}. The latter result and
      \cite[Theorem~6.8]{Crismale-Rossi}  ensure the existence of $\BVB{\mu_k,\nu_k}$
      solutions
      $(\sft_k,\sfq_k)_k$ enjoying the following a priori estimates
      \begin{equation}\label{assum:bound-bis}
\begin{aligned}
\exists\, C>0 \ \ \forall\, k \in \N \, \ \foraa\, s \in (0,S)\, : \qquad &  {\sft}_k'(s) {+}  \| \sfu_k'(s)\|_{W^{1,1}(\Omega)} {+} \|\sfz_k'(s)\|_{\Hs(\Omega)} {+} \|\sfp_k'(s)\|_{L^1(\Omega)} {+} \sqrt{\mu_k} \, \|\sfp_k'(s)\|_{L^2(\Omega)} 
\\
& \quad    
 {+} \|\sfe_k'(s)\|_{L^2(\Omega)}  {+}     \DVito_{\nu_k}(\sfu_k'(s), \sfp_k'(s)) \, \DVito_{\nu_k}^{*,\mu_k}(\sft_k(s),\sfq_k(s))
\leq C
 \end{aligned}
\end{equation}
 (and, up to a reparametrization, the non-degeneracy condition).
 \par
 In Theorem \ref{teo:exparBVsolsingle} below we are going to show that, as  the hardening and rate parameters
 $\mu_k$ and $\nu_k$ vanish, (up to a subsequence)  
 $\BVB{\mu_k,\nu_k}$ solutions converge to a  $\BVZZZ$ solution of system \eqref{RIS-intro}. 
   \begin{theorem}\label{teo:exparBVsolsingle}
Let $(\mu_k)_k,\, (\nu_k)_k$ be two vanishing sequences,
and let $(\sft_k,\sfq_k)_k$ be a sequence of 
$\BVB{\mu_k, \nu_k}$ solutions to system \eqref{RIS-hard-intro}
such that estimate \eqref{assum:bound-bis} holds.
\par
Then,
there exist a (not relabeled) subsequence 
and  a 
curve $ (\sft, \sfq)=(\sft, \sfu, \sfz, \sfp) \in  \eadm 0S0T{\Qpp}$ such that 
items (1), (2), (3) of the statement of  Theorem~\ref{teo:exparBVsol} hold.
\end{theorem}
\begin{proof}
The argument is split in the same steps as the proof of Theorem~\ref{teo:exparBVsol}.
\par
\paragraph{\bf Step $1$: Compactness.} With minor changes,
 from estimate \eqref{assum:bound-bis}
 we derive estimate \eqref{they-stay-in-ball} and 
 convergences \eqref{2906191911} and \eqref{Co-cvg}, whence convergences \eqref{weak-converg-pp}, for the sequence of  $\BVB{\mu_k, \nu_k}$ solutions. 
 Analogously, for the limiting curve $(\sft, \sfq)$ estimate  \eqref{2906191858} holds.
\paragraph{\bf Step $2a$:  energy-dissipation upper estimate when $\sft'>0$} The analogue of 
\eqref{pointwise-endiss} at all $s\in A: = \{ s\in [0,S]\,: \ \sft'(s)>0\}$ can 
be obtained in the same way as in the proof of Theorem~\ref{teo:exparBVsol}, taking into account that
\[
\calM_{0,\mathrm{red}}^{\mu_k,0}(t,q,t',q') = \calM_{0,\mathrm{red}}^{\mu_k,\nu_k}(t,q,t',q') \qquad \text{whenever }t'>0.
\]
 \paragraph{\bf Step $2b$:  energy-dissipation upper estimate when $\sft'=0$}  We will now show that 
\begin{equation}\label{0402221310}
\int_{(0,s) \setminus A}  \calM_{0,\mathrm{red}}^{0,0}(\sft(r),\sfq(r), 0,\sfq'(r)) \dd r   \leq 
  \liminf_{k\to+\infty}
\int_{(0,s) \setminus A} \calM_{0,\mathrm{red}}^{\mu_k,\nu_k}(\sft_{k}(r), \sfq_{k}(r), \sft_{k}'(r), \sfq'_{k}(r)) \dd r.
\end{equation}
 As in the proof  of Thm.\ \ref{teo:exparBVsol}, we  will  apply
 Lemma \ref{l:MIE-ROS} below
in the context of the space $\boldsymbol{Q}: = \Qpp$, with $\boldsymbol{S}$ the ball of radius $R$ from \eqref{they-stay-in-ball}, 
 to the functionals 
$\mathscr{M}_k: = \calM_{0,\mathrm{red}}^{\mu_k,\nu_k}$
 (extended to $\R\times (\Qha{\setminus}\Qpp)\times \R \times (\Qha{\setminus}\Qpp)$ as described for 
$\calM_{0,\mathrm{red}}^{\mu_k,0}$ in the proof of Thm~\ref{teo:exparBVsol}),
 and $\mathscr{M}_0: = \calM_{0,\mathrm{red}}^{0,0}$. 
 With this aim, we only need to check that the
$\Gamma$-liminf estimate in \eqref{G-liminf} holds in our context. Clearly, it is sufficient to check   that 
for any sequence $ (t_k,q_k,t_k',q_k')_k$ with  
$ (t_k,q_k,t_k',q_k') \weaksto (t,q,0,q') \text{ in } \R \times 
  \boldsymbol{S} \times \R \times
  \Qpp  \text{ as } k \to \infty$ there holds 
\begin{equation}\label{0402221311}
  \calM_{0,\mathrm{red}}^{0,0}(t,q,0,q')  \leq 
  \liminf_{k\to+\infty}
 \calM_{0,\mathrm{red}}^{\mu_k,\nu_k}(t_k,q_k,t_k',q_k').
 \end{equation}
 The above estimate easily follows from Lemma \ref{le:2307191000}  in the case in which 
 $(t_k)_k$ admits a strictly positive subsequence. Instead, 
  if there exists $\bar k \in \N$ such that  $t_k' \equiv 0$ for $k \geq \bar k$, then 
 $ \calM_{0,\mathrm{red}}^{\mu_k,\nu_k}(t_k,q_k,t_k',q_k') =  \DVito_{\nu_k}(q_k') \, \DVitoskk(t_k,q_k) $
 for all $k\geq \bar k$ and we may argue in the following way.  
  When $z'=0$ and   $\DVito^*(t,q)=0$  we use that  
\begin{equation*}
 \calM_{0,\mathrm{red}}^{\mu_k,\nu_k}(t_k, q_k, t'_k, q_k') \geq 
\begin{cases}  
\DVito(u_k', p_k') \, \DVito^{*,\mu_k}(t_k,q_k), 
\\
 \|z_k'\|_{L^2}\, \tilded_{L^2(\Omega)} ({-}\mathrm{D}_z \calE_{ \mu_k } (t_k,q_k),\partial\calR(0))\,,
 \end{cases} 
\end{equation*}
 which follows by neglecting some terms in  the expression for $\calM_{0,\mathrm{red}}^{\mu_k,\nu_k}$.
Then, as  in Thm.\ \ref{teo:exparBVsol} we use Lemma~\ref{le:2307191000} to pass to the limit in the two terms on the right-hand side of the above inequality in the cases 
 $z'=0$ and $\DVitosred(t,q)=0$, respectively.
In the remaining case $\|z'\|_{L^2} \DVito^*(t,q)>0$, it holds $\lim_{k \to \infty} \calM_{0,\mathrm{red}}^{\mu_k,\nu_k}(t_k,q_k,t_k',q_k') =+\infty$: indeed, by Lemma~\ref{le:2307191000} and since $z_{k}'\weakto z'$ in~$L^2(\Omega)$, we have that  $\|z'_k\|_{L^2} \DVito^{*,\mu_k}(t_k,q_k)>c$  for a suitable $c>0$. Since 
\[
 \calM_{0,\mathrm{red}}^{\mu_k,\nu_k}(t_{k}, q_{k}, t_{k}', q'_{k}) \geq \frac{1}{\sqrt{\nu_k}} \|z'_k\|_{L^2}   \DVito^{*,\mu_k}(t_k,q_k), 
\] 
we again conclude  estimate \eqref{0402221311}   as $(\mu_k)_k$ and  $(\nu_k)_k$  vanish.

Thus, by Lemma~\ref{l:MIE-ROS}, we have proven \eqref{0402221310}. This finishes the proof.
\end{proof}
%
\medskip
 
\appendix
\section{Some technical results}
We collect the results employed in the proofs of Theorems \ref{teo:exparBVsol} and \ref{teo:exparBVsolsingle}. 
\begin{lemma}\label{lemma:giu}
Let $\sff, \sff_k\,:\ [0,S]\to[0,T]$ be nondecreasing functions such that $\sff_k\to\sff$ uniformly. 
Let $A: = \{ s\in [0,S]\,: \ \sff'(s)>0\}$.
Then for a.a.\ $s\in A$ there is a subsequence $k_j$
and, for every $j$, there is $s_{k_j}\in(0,S)$ such that
$|s_{k_j}{-}s|<\frac1j$ and $\sff'_{k_j}(s_{k_j})>0$.
\end{lemma}

\begin{proof}
Let
\[
\Xi:= \{ s \in A \,:\ \exists\, k(s)\ \forall\, k> k(s) \ \exists\,\delta_k  \ \sff'_k(\sigma)=0 \ \text{for a.a.}\ \sigma\in (s-\delta_k,s+\delta_k) \} .
\]
We shall prove that $\Xi$ is at most countable, which implies the statement of the lemma.
Let $s^1,s^2\in\Xi$. 
Then one has $\sff(s^1)\neq\sff(s^2)$; indeed, 
since $\sff$ is nondecreasing, $\sff(s^1)=\sff(s^2)$ would imply that $\sff$ is constant in $(s^1,s^2)$,
which is in contrast with the assumption $s^1,s^2\in A$ (and thus $\sff'(s^1),\sff'(s^2)>0$).
Let now $y^i_k:=\sff_k(s^i)$ for $i=1,2$.
Since $y^i_k\to\sff(s^i)$ for $i=1,2$, for $k$ sufficiently large one has $|y^1_k-y^2_k|>\frac12|\sff(s^1)-\sff(s^2)|>0$.
Let $\phi_k\in \BV(0,T)$ denote the inverse function of $\sff_k$.
It turns out that $y^1_k,y^2_k$ are both jump points of $\phi_k$ for every $k>\max\{k(s^1),k(s^2)\}$.
Since the jump points of a $\BV$ function are countable, it follows that $\Xi$ is countable, too. 
\end{proof}

  \begin{lemma}\cite[Lemma 7.8]{Crismale-Rossi}\label{le:2307191000}  
 Let $t_k \to t$ in $[0,T]$, $\mu_k \to 0$, $ (q_k)_k=(u_k, z_k, p_k) _k\subset \Qpp$ 
  such that  the following convergences hold as $k\to+\infty$:   $q_k \weaksto   q=(u, z, p)$ in $\Qpp$,  
$e(t_k)=\rmE(u_k+w(t_k))-p_k \to e(t)=\rmE(u+w(t))-p $ in $\Lnn$ and $\mu_k \, p_k \to 0$ in $L^2(\Omega;\MD)$. Then
\begin{subequations}\label{eqs:1106191955'}
\begin{align}
\calS_u \calE_0(t, q)  &\leq \liminf_{ k \to +\infty} \| \mathrm{D}_u \calE_{\mu_k} (t_k, q_k )\|_{( H^1, \bbD )^*} ,\label{1106191956'} 
\\
\tilded_{L^2} ({-}\mathrm{D}_z \calE_0 (t,q ),\partial\calR(0)) & \leq \liminf_{ k \to +\infty} \tilded_{L^2} ({-}\mathrm{D}_z \calE_{\mu_k} (t_k,q_k ),\partial\calR(0)),\label{2307191134}
\\
\calW_p \calE_0(t, q) & \leq \liminf_{ k \to +\infty} \dLtwo ({-}\mathrm{D}_p \calE_{\mu_k} ( t_k,q_k ),\partial_\pi \calH( z_k,0)). \label{1106191958'}
\end{align}
\end{subequations}
\end{lemma}
\par
We borrow our final auxiliary result from \cite{MieRosBVMR}. The proof, therein developed in the case of a sequence $(t_k,q_k)$ with values in $\R \times \boldsymbol{Q}$ with $ \boldsymbol{Q}$ a reflexive space, can be straightforwardly adapted to the case of the dual of a separable space.
\begin{lemma}\cite[Prop.\ 5.2]{MieRosBVMR}
\label{l:MIE-ROS}
Let  $\boldsymbol{Q}$ be the dual of a separable Banach space, let
 $\boldsymbol{S}$  be a weakly$^*$ closed subset  of $\boldsymbol{Q}$, 
 and let $(\mathscr{M}_k)_k, \, \mathscr{M}_0: \R \times \boldsymbol{S} \times \R \times
  \boldsymbol{Q} \to [0,\infty]$ be measurable and weakly$^*$ 
  lower semicontinuous  functionals fulfilling the $\Gamma$-$\liminf$ estimate
  \begin{equation}
  \label{G-liminf}
  \begin{aligned}
  &
  \left((t_k,q_k,t_k',q_k') \weaksto (t,q,t',q') \text{ in } \R \times 
  \boldsymbol{S} \times \R \times
  \boldsymbol{Q}  \text{ as } k \to \infty \right)
  \  \Longrightarrow 
  \ \mathscr{M}_0 (t,q,t',q') \leq \liminf_{k\to \infty} \mathscr{M}_k (t_k,q_k,t_k',q_k').
  \end{aligned}
  \end{equation}
Suppose that,  the functionals $\mathscr{M}_0(t,q,\cdot,\cdot) $ and 
 $\mathscr{M}_k(t,q,\cdot,\cdot) $ are convex for every
 $k\in \N$ and $(t,q) \in \R \times 
  \boldsymbol{S}$. 
  Let $(\sft_k,\sfq_k), \, (\sft,\sfq) \subset \AC ([a,b];\R\times \boldsymbol{S})$ fulfill
  \[
  \sft_k(s) \to \sft(s), \quad \sfq_k(s)\weaksto \sfq(s) \text{ for all } s \in [a,b], 
  \qquad (\sft_k',\sfq_k') \weakto (\sft',\sfq') \text{ in } L^1(a,b;\R\times \boldsymbol{Q}). 
  \]
  Then,
  \[
  \liminf_{k\to\infty} \int_a^b \mathscr{M}_k(\sft_k(s), \sfq_k(s), \sft_k'(s), \sfq_k'(s)) 
  \dd s  \geq \int_a^b \mathscr{M}_0(\sft(s), \sfq(s), \sft'(s), \sfq'(s)) 
  \dd s .
  \]
\end{lemma}

\end{document}